\documentclass[11pt]{amsart}
\usepackage[latin1]{inputenc}
\usepackage[all]{xy}
\usepackage{amssymb, amsmath, amscd, geometry}
\usepackage{graphicx}
\usepackage{setspace}

\numberwithin{equation}{section}
\input xy
\xyoption{all}
\usepackage{latexsym}
\usepackage[usenames]{color}
\usepackage{epic} %%%% Circunferencias grandes y curvas de bezier
\usepackage{mathrsfs}
\usepackage{euscript}
\usepackage{rotating}
\usepackage{bbm}

\usepackage{verbatim}
\usepackage[usenames]{color}

\theoremstyle{plain}
\newtheorem{lemma}{Lemma}[section]
\newtheorem{theorem}[lemma]{Theorem}
\newtheorem{corollary}[lemma]{Corollary}

\newtheorem{prop}[lemma]{Proposition}

\theoremstyle{definition}
\newtheorem{definition}{Definition}[section]
\newtheorem{remark}{Remark}[section]

      \newcommand{\N}{{\mathbb N}}
      \newcommand{\R}{{\mathbb R}}

     \newcommand{\Ne}{{\mathcal N}}
\newcommand{\A}{{\mathcal A}}    
\newcommand{\B}{{\mathcal B}}

\newcommand{\pl }{\,}
\newcommand{\lel }{\, =\, }

\newcommand{\ten}{\otimes}

\newcommand{\uno}{\mathbbm{1}}

\allowdisplaybreaks

\begin{document}

\baselineskip=17pt

\title[Channel capacities via $p$-summing norms]{Channel capacities via $p$-summing norms}

\author{Marius Junge and Carlos Palazuelos}

\thanks{The first author is partially supported by NSF DMS-1201886. The second author is partially supported by the EU grant QUEVADIS, Spanish projets QUITEMAD, MTM2011-26912 and  the ``Juan de la Cierva'' program. Both authors are partially supported by MINECO: ICMAT Severo Ochoa project SEV-2011-0087.}

\maketitle

\begin{abstract}
In this paper we show how \emph{the metric theory of tensor products} developed by Grothendieck perfectly fits in the study of
channel capacities, a central topic in \emph{Shannon's information theory}. Furthermore, in the last years Shannon's theory has been
generalized to the quantum setting to let the \emph{quantum information theory} step in. In this paper we consider the classical
capacity of quantum channels with restricted assisted entanglement. In particular these capacities include the classical capacity and the
unlimited entanglement-assisted classical capacity of a quantum channel. To deal with the quantum case we will use the noncommutative version of $p$-summing maps. More precisely, we prove that the (product state) classical capacity of a quantum channel with restricted assisted entanglement can be expressed as the derivative of a completely $p$-summing norm.
\end{abstract}
\section{Introduction}\label{Section Introduction}
In the late 1940s Shannon single-handedly established the entire mathematical field of information theory in his famous paper \emph{A Mathematical Theory of Communication} (\cite{Shannon}). Some ground-breaking ideas like the quantization of the information content of a message by the Shannon entropy, the concept of channel capacity or the schematic way to understand a communication system were presented in \cite{Shannon}, laying down the pillars of the future research in the field. Being naturally modeled by a stochastic action, a \emph{noisy channel} is defined as a (point-wise) positive linear map $\Ne:\R^n_A\rightarrow \R^n_B$ between the sender (Alice) and the receiver (Bob) which preserves probability distributions. In terms of notation, we will denote a channel by $\Ne: \ell_1^n\rightarrow \ell_1^n$\footnote{Note that a channel acting on $n$-bit strings should be denoted by $\Ne: \ell_1^{2^n}\rightarrow \ell_1^{2^n}$}. Shannon defined the \emph{capacity of a channel} as an asymptotic ratio\footnote{We will give a formal definition below.}: $$\frac{\text{number of transmitted bits with an $\epsilon \rightarrow 0$ error}}{\text{number of required uses of the channel in parallel}}.$$
One of the most important results presented in \cite{Shannon} is the so called \emph{noisy channel coding theorem}, which states that for every noisy channel $\Ne:\ell_1^n \rightarrow\ell_1^n$ its capacity is given by
\begin{align}\label{classical-classical- Equation}
C_c(\Ne)=\max_{P=(p(x))_x}H(X:Y),
\end{align}where $H(X:Y)$ denotes the mutual information\footnote{$H(X:Y)=H(X)+H(Y)-H(X,Y)$, where $H$ represents the Shannon entropy.} of an input distributions $P=(p(x))_x$ for $X$ and the corresponding induced distribution at the output of the channel $(\Ne(P))_y$.
Although our main Theorem \ref{mainI} will be stated in a much more general context, it already uncovers a beautiful relation between Shannon information theory and $p$-summing maps when it is applied to classical channels. Indeed, it states that for every channel $\Ne:\ell_1^n\rightarrow \ell_1^n$ we have
\begin{align}\label{mainI-classical channels}
C_c(\Ne)=\frac{d}{dp}\big[\pi_q(\Ne^*)\big]|_{p=1},
\end{align}where $\pi_q(\Ne^*)$ denotes the $q$-summing norm of the map $\Ne^*:\ell_\infty^n\rightarrow \ell_\infty^n$ and $\frac{1}{p}+\frac{1}{q}=1$.

In the very last years, Shannon's theory has been generalized to the quantum setting. In this new context, one replaces probability distributions by density operators: Semidefinite positive operators $\rho$ of trace one; so the natural space to work with is $S_1^n$\footnote{$S_1^{2^n}$ if we are dealing with $n$ quantum bits or \emph{qubits}.} (the space of trace class operators). Then, we define a \emph{quantum channel} as a completely positive\footnote{The requirement of completely positivity is explained by the fact that our map must be a channel when we consider our system as a physical subsystem of an amplified one (with an environment) and we consider the map $1_{Env}\otimes \Ne$.} and trace preserving linear map on $M_n$. Analogously to the classical case, we will denote a quantum channel by $\Ne:S_1^n\rightarrow S_1^n$.

Quantum information becomes particularly rich when we deal with bipartite states thanks to \emph{quantum entanglement}. Entanglement is a fundamental resource in quantum information and quantum computation and it is not surprising that it plays a very important role in the study of channels. In particular, it can be seen that the capacity of a quantum channel can be increased if the sender and the receiver are allowed to use a shared entangled state in their protocols. In this work we will study the capacity of a quantum channel to transmit classical information; that is, the \emph{classical capacity}. However, we can consider different classical capacities depending on the amount of shared entanglement allowed to Alice and Bob. Given a quantum channel $\Ne:S_1^n\rightarrow S_1^n$ we will call \emph{$d$-restricted classical capacity of $\Ne$} to the classical capacity of the channel when Alice and Bob are allowed to use a $d$-dimensional entangled state per channel use in the protocol. In fact, our capacity is very closely related to the one studied in \cite{Shor}, where the author imposed the restriction on the entropy of entanglement per channel use. We will explain the connections between the two definitions in Section \ref{Phy Int}. Therefore, we define a family of capacities such that for the case $d=1$ we recover the so called \emph{classical capacity of $\Ne$ (without entanglement)}, $C_c(\Ne)$, and taking the supremum on $d\geq 1$ we obtain the so called \emph{(unlimited) entanglement-assisted classical capacity of $\Ne$}, $C_E(\Ne)$. This family of capacities can be defined within the following common ratio-expression:
\begin{align*}
\lim_{\epsilon\rightarrow 0}\limsup_{k\rightarrow \infty}\Big\{\frac{m}{k}:\exists_\mathcal A,\exists_\mathcal B \text{   such that   } \|id_{\ell_1^{2^m}}-\mathcal B\circ \Ne^{\otimes_k}\circ \mathcal A\|< \epsilon\Big\}.
\end{align*}
Here $\mathcal A$ and $\mathcal B$ represent Alice's encoder and Bob's decoder channels respectively (which will depend on the resources they can use in their protocol) and $\Ne^{\otimes_k}$ denotes the $k$ times uses of the channel in parallel. The reader will find a more extended explanation about the different classical capacities of a quantum channel in Section \ref{Phy Int}.

In order to compute the classical capacities of a quantum channel $\Ne$ one could expect to have an analogous result to (\ref{classical-classical- Equation}). However, the situation is more difficult in the case of quantum channels. A first approach to the problem consists of restricting the protocols that Alice and Bob can perform. We will talk about the \emph{product state version} of a capacity when we impose that Alice (the sender) is not allowed to distribute one entangled state among more than one channel use. More specifically, for any quantum channel $\Ne:S_1^n\rightarrow S_1^n$ and any $1\leq d\leq n$ let us denote by $C_{prod}^d(\Ne)$ the classical capacity of $\Ne$ with assisted entanglement when\footnote{We will explain the quantity $C_{prod}^d(\Ne)$ in more detail in Section \ref{Phy Int}.}
\begin{enumerate}
\item[a)] Alice and Bob are restricted to protocols in which they start sharing a (pure) $d$-dimensional bipartite state per channel use.
\item[b)] The sender inputs one and only one of (their part of) these entangled states in each channel.
\end{enumerate}
Following the same ideas as in \cite{BSST} and \cite{Shor} one can see 
\begin{align}\label{d-restricted capacity}
C_{prod}^d(\Ne):=\sup\Big\{S\Big(\sum_{i=1}^N \lambda_i (\Ne\circ \phi_i)((tr_{M_d}\otimes id_{M_d})(\eta_i))\Big)
\\ \nonumber+\sum_{i=1}^N \lambda_i\Big[S\Big((id_{M_d}\otimes tr_{M_d})(\eta_i)\Big)-S\Big(\big(id_{M_d}\otimes (\Ne\circ \phi_i)\big)(\eta_i)\Big)\Big] \Big\}.
\end{align}Here, $S(\rho):= -tr(\rho \log_2 \rho)$ denotes the \emph{von Neumann entropy of a quantum state} $\rho$ and the supremum runs over all $N\in \N$, all probability distributions $(\lambda_i)_{i=1}^N$, and all families $(\phi_i)_{i=1}^N$, $(\eta_i)_{i=1}^N$, where $\phi_i:S_1^d\rightarrow S_1^n$ is a quantum channels and $\eta_i\in S_1^d\otimes S_1^d$ is a pure state for every $i=1,\cdots, N$.

Equation (\ref{d-restricted capacity})  reduces to Equation (\ref{classical-classical- Equation}) when $\Ne$ is a classical channel. On the other hand, it can bee seen that in the case $d=1$ we recover the Holevo-Schumacher-Westmoreland Theorem, which describes the product state classical capacity (or Holevo capacity) of a quantum channel (\cite{Hol}, \cite{SW}). Moreover, in the case $d=n$ we recover the Bennett-Shor-Smolin-Thapliyal Theorem, which gives a formula to compute the entanglement-assisted classical capacity of a quantum channel (\cite{BSST}). The reader will find a brief introduction to these capacities in Section \ref{Phy Int}. In order to obtain the general capacities (rather than the product state version) one has to consider the corresponding regularization. It is not difficult to see that in this case the regularization is given by
\begin{align*}
C^d(\Ne)=\sup_k\frac{C_{prod}^{d^k}(\otimes^k \Ne)}{k}.
\end{align*}
Recently Hastings solved a long-standing open question in quantum information theory by showing that $C^1(\Ne)\neq C_{prod}^1(\Ne)$ for certain quantum channels (\cite{Has}). Hastings' result shows that we do need to consider the regularization of $C_{prod}^1(\Ne)$ to compute the classical capacity and it is not enough to consider the much easier formula given in (\cite{Hol}, \cite{SW}). On the other hand, one can check the the formulae (\ref{classical-classical- Equation}) and the one given in \cite{BSST} to express the product state version of the unlimited entanglement-assisted classical capacity of a quantum channel are additive on channels: $C(\Ne_1\otimes \Ne_2)=C(\Ne_1)+C(\Ne_2)$. This is a crucial fact since in these cases, no regularization of the product state version of the corresponding capacities is required and, thus, those formulae describe the general capacities.

Equation (\ref{d-restricted capacity}) expresses mathematically the capacity of a channel, which was previously defined by means of concepts like \emph{protocols} or \emph{many uses of the channel in parallel}. The main result presented in this work shows a direct connection between the quantity $C_{prod}^d(\Ne)$ and the theory of \emph{absolutely $p$-summing maps}. Introduced first by Grothendieck in \cite{Grothendieck}, the theory of $p$-summing maps was exhaustively studied by Pietsch (\cite{Pie}) and Lindenstrauss and Pelczynski (\cite{LiPe}). In fact, it was in this last seminal work where the authors showed the extreme utility of $p$-summing maps in the study of many different problems in Banach space theory. We recommend the references \cite{DeFl} and \cite{DJT} for a complete study on the topic. The generalization of the theory of absolutely $p$-summing maps to the noncommutative setting was developed by Pisier by means of the so called \emph{completely $p$-summing maps} (\cite{Pisierbook2}). Even generalizing the definition of $p$-summing maps to the noncommutative setting is not obvious since it requires the concept of noncommutative vector valued $L_p$-spaces. However, absolutely $p$-summing maps admit another natural generalization to the so called \emph{$(cb,p)$-summing maps}, introduced by the first author (\cite{Junge-Hab}), which can be seen as a generalization to an intermediate setting between the Banach space case and the completely $p$-summing maps. In a complete general way we will consider here the $\ell_p(S_p^d)$-maps which include, in particular, the two previous definitions. Thanks to the factorization theorem proved by Pisier (\cite[remark 5.11]{Pisierbook2}) we have the following easy definition for maps defined on $M_n$:
\begin{align*}
\pi_{q,d}(T:M_n\rightarrow M_n)=\inf\Big\{\|id_{M_d}\otimes \tilde{T}\|_{M_d(S_p^n)\rightarrow M_d(M_n)}: T=\tilde{T}\circ M_{a,b}\Big\},
\end{align*}
where the infimum runs over al factorizations of $T$ with $a,b\in M_n$ verifying $\|a\|_{S_{2p}^n}=\|b\|_{S_{2p}^n}=1$, $M_{a,b}:M_n\rightarrow S_p^n$ being the linear map defined by $M_{a,b}(x)=axb$ for every $x\in M_n$ and $\tilde{T}:S_p^n\rightarrow M_n$ being a linear map.

Our main result states as follows.
\vspace{-0.15 cm}
\begin{theorem}\label{mainI}
Given a quantum channel $\Ne:S_1^n\rightarrow S_1^n$,
\begin{equation}\label{Main result restricted capacity}
C_{prod}^d(\Ne)= \frac{d}{dp}\big[\pi_{q,d}(\Ne^*)\big]|_{p=1},
\end{equation}where $\frac{1}{p}+\frac{1}{q}=1$. Here, $\pi_{q,d}(\Ne^*)$ denotes the $\ell_q(S_q^d)$-summing norm of $\Ne^*:M_n\rightarrow M_n$.
\end{theorem}
Actually, to have the equality (\ref{Main result restricted capacity}) we must define $C_{prod}^d(\Ne)$ (\ref{d-restricted capacity}) by using the \emph{$\ln$-entropy}, $S(\rho):= -tr(\rho \ln \rho)$, instead of using $\log_2$ as it is usually done in quantum information. However, since both definitions are the same up to a multiplicative factor, we could use the standard entropy $S$ and we should then write (\ref{Main result restricted capacity}) as $C_{prod}^d(\Ne)= \frac{1}{\ln 2}\frac{d}{dp}\big[\pi_{q,d}(\Ne^*)\big]|_{p=1}$. In this work we will always consider $\ln$-entropies to avoid this constant factor.  

As we mentioned before, Hastings' result says that we cannot avoid the regularization
of $C_{prod}^d(\Ne)$ if $d=1$. We will show that the additivity of $C_{prod}^d$ has a particularly bad behavior for $1< d< n$. Indeed, we will prove the following strong non-additive result.
\vspace{-0.15 cm}
\begin{theorem}\label{counterexample additivity}
There exists a channel $\Ne:S_1^{2n}\rightarrow S_1^{2n}$ such that $$C_{prod}^{n}(\Ne\otimes \Ne)\succeq \frac{1}{3}\ln n+ 2C_{prod}^{\sqrt{n}}(\Ne),$$where we use the symbol $\succeq$ to denote inequality up to universal (additive) constants which do not depend on $n$.
\end{theorem}
Theorem \ref{counterexample additivity} says that the general $d$-restricted capacity with $1<d< n$ can be, in fact, very different from $C_{prod}^d$ (the product state version). Nevertheless, we should emphasize that the nature of the non additivity of $C_{prod}^d$ with $1<d< n$ comes from the fact that one must change the entanglement dimension from $d$ to $d^2$ when one considers the tensor product of two channels. This makes the problem of additivity (so the regularization) completely different from the much deeper case $d=1$.

The paper is organized as follows. In the first section we briefly introduce the notion of noncommutative $L_p$ spaces and $\ell_p(S_p^d)$-summing maps. Furthermore, we prove a modified version of Pisier's theorem in order to have a more accurate result for the particular maps that we are considering in this work. In Section \ref{main results section} we give the proof of our main result, Theorem \ref{mainI}, and we explain how to obtain the particular cases commented above. In Section \ref{Section: non-additivity} we explain why the $d$-restricted capacity is easier to compute when we deal with covariant channels and we use this fact to prove Theorem \ref{counterexample additivity}. Finally, in Section \ref{Phy Int} we give an explanation of the restricted classical capacities of quantum channels and we state some of the most important results in the area. We also discuss the physical interpretation of the $C_{prod}^d$ capacity and the connections with some capacities previously studied in \cite{Shor}.
\section{Pisier's theorem for quantum channels}\label{Pisier's theorem Section}
Following the metric theory of tensor product developed first by Grothendieck and subsequently by Pietsch, Lindenstrauss and Pelczynski in terms of $p$-summing maps, in \cite{Pisierbook2} Pisier introduced the notion of completely $p$-summing map between operator spaces. Pisier showed a satisfactory factorization theorem for these kinds of maps, analogous to the existing result in the commutative setting. In this section we will study such a factorization theorem when it is applied to completely positive maps and we will show that in this case one can get some extra properties in the statement of the theorem. Furthermore, in order to define our restricted capacities, we will need to consider the more general $\ell_p(S_p^d)$- summing maps. For the sake of completeness we will start with a brief introduction to noncommutative (vector valued) $L_p$-spaces and completely $p$-summing maps. Since we will restrict our work to finite dimensional von Neumann algebras, we will mainly focus on this setting. However, the theory of noncommutative $L_p$-spaces has been developed in a much more general context and most of the results can be stated in such a general framework. We recommend \cite{Pisierbook2}, \cite{PiXu} for a complete study of the subject. Since the key point to define noncommutative (vector valued) $L_p$-spaces is to consider operator spaces, we will assume the reader to be familiarized with them. We recommend \cite{Pisierbook} for the non familiar reader with the topic.

Let $\A$ be a hyperfinite von Neumann algebra equipped with a faithful normal semi-finite trace $\varphi$. Let us denote $L_\infty(\A):=\A$ and $L_1(\A):=\A_{*}$, where $\A_{*}$ is the predual of $\A$ (with respect to $\varphi$). We recall that $L_1(\A)$ can be described as the completion of the linear space $\{x\in \A: \|x\|_1:=\varphi(|x|)<\infty\}$ (see \cite[Proposition 2.19]{Takesaki}). Then, one can use complex interpolation to define the noncommutative $L_p$-space  $L_p(\A):=[L_\infty(\A),L_1(\A)]_{\frac{1}{p}}$. In the particular case $\A=M_d$ (and $\varphi=tr_{M_d}$) we write $L_p(M_d)=S_p^d$ for $1\leq p< \infty$ and $L_\infty(M_d)=M_d$. Given a linear map $T:L_q(\A)\rightarrow L_p(\B)$, we denote the operator norm by
\begin{align*}
\|T\|=\sup_{A\in \A} \frac{\| T(A)\|_p}{\|A\|_q}.
\end{align*}
In the following, we will just write $id_n$ to denote $id_{M_n}:M_n\rightarrow M_n$. The Banach space $L_p(\A)$ can be endowed with an operator space structure (o.s.s). We regard $\A$ as a subspace of $B(\mathcal H)$ with $\mathcal H$ being the Hilbert space arising from the GNS construction. On the other hand, we can also embed the predual von Neumann algebra $\A_{*}$ on its bidual $\A^*$, to obtain an o.s.s. for $\A_{*}$. The o.s.s. on $L_1(\A)$ is then given by that of $\A_{*}^{op}$. We refer to \cite[Chapter 7]{Pisierbook} for a detailed justification of this definition. Then, the complex interpolation for operator spaces provides a natural o.s.s. on $L_p(\A):=[L_\infty(\A),L_1(\A)]_{\frac{1}{p}}$. The definition of an o.s.s. on $L_p(\A)$ allows us to talk about the completely bounded norm of a map
$T:L_q(\A)\rightarrow L_p(\B)$.  In general, given two operator spaces $E$ and $F$, we define
\begin{align}\label{CB}
\|T\|_{cb}=\sup_{d\in \N}\big\|id_d\otimes T: M_d(E)\rightarrow M_d(F)\big\|
\end{align} or, equivalently,
\begin{align*}
\|T\|_{cb}=\sup_{d\in \N}\Big(\sup_Y\frac{\big\|(id_d\otimes T)(Y)\big\|_{M_d(F)}}{\|Y\|_{M_d(E)}}\Big).
\end{align*}In our particular case, it can be seen (\cite[Lemma 1.7]{Pisierbook2}) that
\begin{align*}
\|Y\|_{M_d(L_p(\A))}=\sup_{A,B\in B_{S_{2p}^d}}\big\|(A \otimes \uno_{\A})Y(B\otimes \uno_{\A})\big\|_{L_p(M_d\otimes_{min}\A)}.
\end{align*}Here, $B_{S_{2p}^d}$ denotes the unit ball of the $2p$ - Schatten class of operators in $M_d$ and $\uno_\A$ denotes the identity of the von Neumann algebra $\A$. Sometimes we will just write $\uno$. 

For our purpose we need to introduce the noncommutative vector valued $L_p$-spaces. We will restrict here to the discrete case because it is the one we will use in this work. We refer to \cite{Pisierbook2} for the more general case of (finite) injective Neumann algebras. Given a Hilbert space $\mathcal  H$, the space of compact operators on $\mathcal H$, $\mathcal K(\mathcal H):=S_\infty(\mathcal H)$, can be endowed with a natural o.s.s. via its natural inclusion on $B(\mathcal H)$, the space of bounded operators on $\mathcal H$. As we explained before, we ca also endow the predual of $B(\mathcal H)$, $S_1(\mathcal H)$, with a natural o.s.s. It is well known that this space is the trace class of operators acting on $\mathcal H$ and it does coincide with the dual of $S_\infty(\mathcal H)$. Then, by complex interpolation we can define an o.s.s. on $S_p(\mathcal H):=[S_\infty(\mathcal H), S_1(\mathcal H)]_{\frac{1}{p}}$. In our case, we will restrict to $\mathcal H=\ell_2$ (or $\mathcal H=\ell_2^n$) and we will just write $S_p(\ell_2)=S_p$ (or $S_p(\ell_2^n)=S_p^n$) for every $1\leq p\leq \infty$. Given any operator space $E$, we will denote $S_\infty[E]=S_\infty\otimes_{min} E$, where $\min$ denotes the minimal tensor norm in the category of operator spaces. On the other hand, Effros and Ruan introduced the space $S_1[E]$ as the (operator) space $S_1\widehat \otimes E$, where $\widehat \otimes$ denotes the projective operator space tensor norm. Then, using complex interpolation Pisier defined the noncommutative vector valued (operator) space $S_p[E]=\big[S_\infty[E], S_1[E]\big]_{\frac{1}{p}}$ for any $1\leq p\leq \infty$ and he proved that this definition leads to obtain the expected properties of $S_p[E]$, analogous to the commutative setting (see \cite[Chapter 3]{Pisierbook2}). We denote $S_p^n[E]=\big[S_\infty^n[E], S_1^n[E]\big]_{\frac{1}{p}}$. One can also check (\cite[Theorem 1.5]{Pisierbook2}) that for every $1\leq p<\infty$ and any operator space $E$, the norm of an element $X\in S_p[E]$ verifies
\begin{align}\label{Norm in S_p[E]}
\|X\|_{S_p[E]}=\inf \big\{\|A\|_{S_{2p}}\|Y\|_{B(\ell_2)\otimes_{\min}E}\|B\|_{S_{2p}}\big\},
\end{align}where the infimum runs over all representations of the form $X=\big(A\otimes \uno_{B(\ell_2)}\big)Y\big(B\otimes \uno_{B(\ell_2)}\big)$. We have an analogous formula for $\|X\|_{S_p^n[E]}$. In this work we will mainly deal with the case $E=S_q^d$ for some $1\leq q\leq \infty$. It can be seen that, given $1\leq p,q\leq \infty$ and defining $\frac{1}{r}=|\frac{1}{p}-\frac{1}{q}|$, we have:

\noindent If $p\leq q$,
\begin{align}\label{norm p< q}
\|X\|_{S_p^n[S_q^d]}=\inf\Big\{\|A\|_{S_{2r}^n}\|Y\|_{S_q^{nd}}\|B\|_{S_{2r}^n}\Big\},
\end{align}where the infimum runs over all representations $X=(A\otimes \uno_{M_d})Y(B\otimes \uno_{M_d})$ with $A,B\in M_n$ and $Y\in M_n\otimes M_d$.

\noindent If $p\geq q$,
\begin{align}\label{norm p>q}
\|X\|_{S_p^n[S_q^d]}=\sup\Big\{\big\|(A\otimes \uno_{M_d})X(B\otimes \uno_{M_d})\big\|_{S_{q}^{nd}}: A,B\in B_{S_{2r}^n}\Big\}.
\end{align}
As an interesting application of this expression for the norm in $S_p[S_q]$ in \cite[Theorem 1.5 and Lemma 1.7]{Pisierbook2} Pisier showed that for a given linear map between operator spaces $T:E\rightarrow F$ we can compute its completely bounded norm as
\begin{align*}
\|T\|_{cb}=\sup_ {d\in \N}\big\|id_d\otimes T:S_t^d[E]\rightarrow S_t^d[F]\big\|
\end{align*}for every $1\leq t\leq \infty$. That is, we can replace $\infty$ in (\ref{CB}) with any $1\leq t\leq \infty$ in order to compute the cb-norm.
\begin{remark}\label{(p,1)}
It is known (\cite{War}, \cite{Aud2}) that if $T$ is completely positive we can compute $\|T:S_q\rightarrow S_p\|$ by restricting to positive elements $A\in S_q$. Moreover, in this case one can also consider positive elements $X\geq 0$ to compute the cb-norm $\|T\|_{cb}=\|id_{S_q}\otimes T:S_q[S_q]\rightarrow S_q[S_p]\|$ (\cite[Section 3]{DJKR}). On the other hand, given a positive element $X\geq 0$, one can consider $A=B> 0$ in the expressions (\ref{norm p< q}) and  (\ref{norm p>q}) for $\|X\|_{S_p^n[S_q^d]}$. According to this, if $X\geq 0$ and $q=1$, (\ref{norm p>q}) becomes
\begin{align*}
\|X\|_{S_p^n[S_1^d]}=\sup_{A>0}\frac{\|(A\otimes \uno_{M_d})X(A\otimes \uno_{M_d})\|_{S_1^{nd}}}{\|A\|_{2p'}^2}
=\|(id_n\otimes tr_d)(X)\|_p,
\end{align*}where $\frac{1}{p}+\frac{1}{p'}=1$. Here and in the rest of the work we use notation $tr_n:=tr_{M_n}$.
\end{remark}
A linear map between operator spaces $T:E\rightarrow F$ is called \emph{completely $p$-summing} if
\begin{align}\label{Def completely p-summing}
\pi_p^o(T):=\pi_{p,\infty}(T)=\big\|id_{S_p}\otimes T:S_p\otimes_{min} E\rightarrow S_p[F]\big\|< \infty.
\end{align}Note that we can write, equivalently,
\begin{align*}
\pi_p^o(T):=\sup_d\pi_p^d(T),
\end{align*}where $\pi_p^d(T)=\big\|id_d\otimes T:S_p^d\otimes_{min} E\rightarrow S_p^d[F]\big\|.$
This definition generalizes the absolutely $p$-summing maps defined in the Banach space category. In \cite{Pisierbook2} Pisier proved that most of the properties of $p$-summing maps have an analogous statement in this noncommutative setting. In particular, it can be seen that the completely $p$-summing maps verify a satisfactory Pietsch factorization theorem (\cite[Theorem 5.1]{Pisierbook2}). The theory of completely $p$-summing maps becomes particularly nice when we consider the case $E=F=M_n$. Then, the definition of the completely $p$-summing norm of the map $T:M_n\rightarrow M_n$ can be stated as
\begin{align*}
\pi_p^o(T):=\sup_d\big\|(id_d\otimes T)\circ \text{flip}:M_n(S_p^d)\rightarrow S_p^d[M_n]\big\|,
\end{align*}where the $\text{flip}$ operator is defined as $\text{flip}(a\otimes b)=b\otimes a$.
Pietsch factorization theorem is particularly simple in this case and has a complete analogous statement to the commutative result (see \cite[Theorem 5.9 and Remark 5.10]{Pisierbook2}). In particular, one can deduce
\begin{align*}
\pi_p^o(T:M_n\rightarrow M_n)=\pi_p^n(T:M_n\rightarrow M_n),
\end{align*} and
\begin{align*}
\pi_p^o(T:M_n\rightarrow M_n)=\sup\Big\{\big|tr(S\circ T)\big|:\pi_q^o(S:M_n\rightarrow M_n)\leq 1\Big\},
\end{align*}where $\frac{1}{p}+\frac{1}{q}=1$. This last assertion follows from the duality theorem proved in \cite[Corollary 3.1.3.9]{Junge-Hab} and the fact that for maps $S:M_n\rightarrow M_n$ the completely $q$-summing norm and the $q$-nuclear norm coincide. In fact, it is very easy to extend this result to maps $S:\ell_\infty^N(M_{k_i})\rightarrow M_n$ as follows:
\begin{align}\label{duality p-summing}
\pi_p^o(T:M_n\rightarrow \ell_\infty^N(M_{k_i}))=\sup\Big\{\big|tr(S\circ T)\big|:\pi_q^o(S:\ell_\infty^N(M_{k_i})\rightarrow M_n)\leq 1\Big\},
\end{align} where $\ell_\infty^N(M_{k_i}):=\bigoplus_{i=1}^NM_{k_i}$.

As we will explain later in detail, in order to consider a general family of restricted capacities we will have to deal with the completely $p$-summing norm of maps defined between finite dimensional von Neumann algebras. Therefore, we will need to adapt Pisier's factorization theorem for completely $p$-summing maps to our particular context. Actually, due to the fact that we will consider quantum channels, we will state such a factorization theorem for these particular maps obtaining some extra properties. We start with the following result, which is essentially proved in \cite{DJKR}. We add here the proof for the sake of completeness.
\begin{prop}\label{CS CP}
Let $\mathcal A$ and $\mathcal B$ be finite dimensional von Neumann algebras and let $T:\A\rightarrow \B$ be a completely positive map. Then, for every $1\leq q\leq \infty$ and any pair of elements $x,y\in S_p[\mathcal A]$ we have
\begin{align*}
\big\|(id_{S_p}\otimes T)(xy)\big\|_{S_p[\mathcal B]}\leq \big\|(id_{S_p}\otimes T)(xx^*)\big\|^\frac{1}{2}_{S_p[\mathcal B]}\big\|(id_{S_p}\otimes T)(y^*y)\big\|^\frac{1}{2}_{S_p[\mathcal B]}.
\end{align*}
\end{prop}
\begin{proof}
It suffices to show the result for $x,y\in S_p^N[\mathcal A]$, with $N$ arbitrarily large. 
Let us simplify notation by writing $\tilde{T}=id_{S_p^N}\otimes T$. Given the elements $x,y\in S_p^N[\mathcal A]$, we consider the positive element
$$z= \left(%
\begin{array}{c}
 x  \\
 y^*  \\
\end{array}%
\right)\left(%
\begin{array}{cc}
 x^* & y 
\end{array}%
\right)=\left(%
\begin{array}{cc}
 xx^* & xy \\
 (xy)^* & y^*y \\
\end{array}%
\right)\in M_2(S_p^N[\mathcal A]).$$ 
Since $T$ is completely positive, we know that 
$$(id_{M_2}\otimes \tilde{T})(z)= \left(%
\begin{array}{cc}
 \tilde{T}(xx^*) & \tilde{T}(xy) \\
(\tilde{T}( (xy))^* & \tilde{T}(y^*y) \\
\end{array}%
\right)\in M_2(S_p^N[\mathcal B])$$is a positive element. 
According to \cite[Lemma 3.5.12]{HoJo} this implies that there exists a contraction $C\in M_N\otimes_{min} \mathcal B$ so that $$\tilde{T}(xy)=\tilde{T}(xx^*)^\frac{1}{2}C\tilde{T}(y^*y)^\frac{1}{2}.$$In fact, \cite[Lemma 3.5.12]{HoJo} is stated for $M_N\otimes M_K$, but one can extend the result to finite dimensional von Neumann algebras in a straightforward manner (see also Remark \ref{General Proposition 2.1}). Then, we have that
\begin{align*}
\big\|\tilde{T}(xy)\big\|_{S_p^N[\mathcal B]}=\big\|\tilde{T}(xx^*)^\frac{1}{2}C\tilde{T}(y^*y)^\frac{1}{2}\big\|_{S_p^N[\mathcal B]}\leq \big\|\tilde{T}(xx^*)\big\|^\frac{1}{2}_{S_p^N[\mathcal B]}\big\|\tilde{T}(y^*y)\big\|^\frac{1}{2}_{S_p^N[\mathcal B]},
\end{align*}where the last inequality follows from \cite[Lemma 9]{DJKR} with $p=\infty$.
This concludes the proof.
\end{proof}
Now, we state the main result of this section. In the following, we will write $\|\cdot\|_+$ to denote the corresponding norm $\|\cdot\|$ when we restrict to positive elements.
\begin{theorem}\label{factorization-positivity}
Let $\A$ and $\B$ be finite dimensional  von Neumann algebras and let $T:\A\rightarrow \B$ be a completely positive map. Then, the following assertions are equivalent:
\begin{enumerate}
\item[a)] $\pi_p^o(T)\leq C$.
\item[b)] $\big\|id_{S_p}\otimes T:S_p\otimes_{min} \A\rightarrow S_p[\B]\big\|_+\leq C$.
\item[c)]  There exists a positive element $a\in \A$ verifying $\|a\|_{L_{2p}(\A)}\leq1$ such that for every $x\in S_p\otimes_{min} \A$ we have $$\big\|\big(id_{S_p}\otimes T\big)(x)\big\|_{S_p[\B]}\leq C\big\|(\uno\otimes a)x(\uno\otimes a)\big\|_{S_p[L_p(\A)]}.$$
\item[d)] There exist a positive element $a\in \A$ verifying $\|a\|_{L_{2p}(\A)}\leq1$ and a completely positive linear map $\alpha:L_p(\A)\rightarrow \B$ such that $T=\alpha\circ M_{a}$ and $\|\alpha\|_{cb}\leq C$.
\noindent Here $M_{a}:\A\rightarrow L_p(\A)$ is the linear map defined by $M_a(x)=axa$ for every $x\in \A$.
\end{enumerate}
Furthermore, $\pi_p^o(T)=\inf\Big\{C: C \text{   verifies any of the above conditions}\Big\}.$
\end{theorem}
The proof of Theorem \ref{factorization-positivity} is based on a slight modification of the Hahn-Banach argument used in (\cite[Theorem 5.1]{Pisierbook2}) and a non trivial use of Cauchy-Schwartz inequality. Since Pisier's theorem has not been stated for finite dimensional von Neumann algebras and the positivity is always tricky in these contexts we will explain the proof in detail.
\begin{proof}
Since $\A$ is a finite dimensional von Neumann algebra, we can assume $\A=\bigoplus_{i=1}^NM_{k_i}:=\ell_\infty^N(M_{k_i})$. 
The proof of $a) \Rightarrow b)$ is trivial. The implications $c)\Rightarrow d)$ and $d)\Rightarrow a)$ follow by standard arguments. So, we have to show $b)\Rightarrow c)$.

By assumption, for every $x_1,\cdots ,x_m$ positive elements in $S_p\otimes \ell_\infty^N(M_{k_i})$ we have
\begin{align*}\sum_{j=1}^m\big\|(id_{S_p}\otimes T)(x_j)\big\|_ {S_p[\B]}^p\leq C^p\sup_i\sup_{a_i,b_i}\sum_{j=1}^m\big\|(\uno \otimes a_i)x_j(i)(\uno\otimes b_i)\big\|_{S_p(\ell_2\otimes \ell_2^{k_i})}^p,
\end{align*}where $x_j=(x_j(i))_{i=1}^N\in \ell_\infty^N(S_p\otimes M_{k_i})$ for every $j=1,\cdots ,m$ and the supremum on the right hand side is taken over all $a_i$ and $b_i$ positive elements in the unit ball of $S_{2p}^{k_i}$ for every $i=1,\cdots ,N$. Furthermore, as a consequence of the noncommutative generalized H\"older's inequality (see for instance \cite[Section 1]{PiXu}) we deduce that for every $x_1,\cdots ,x_m$ positive elements in $S_p\otimes \ell_\infty^N(M_{k_i})$ we have
\begin{align*}
\sum_{j=1}^m\big\|(id_{S_p}\otimes T)(x_j)\big\|_ {S_p[\B]}^p\leq C^p\sup_i\sup_{a_i}\sum_{j=1}^m\big\|(\uno \otimes a_i)x_j(i)(\uno\otimes a_i)\big\|_{S_p(\ell_2\otimes \ell_2^{k_i})}^p,
\end{align*}where the sup is taken over $a_i$ positive elements in the unit ball of $S_{2p}^{k_i}$.

Following a Hahn-Banach argument as in (\cite[Theorem 5.1]{Pisierbook2}) we can conclude the existence of a sequence of positive numbers $(\lambda_i)_{i=1}^N$ verifying $\sum_{i=1}^N\lambda_i^p=1$ and a sequence of positive elements $a_i$ in the unit ball of $S_{2p}^{k_i}$ such that
\begin{align}\label{Pietsch positive elements}
\big\|(id_{S_p}\otimes T)(x)\big\|_ {S_p[\B]}^p\leq C^p\sum_{i=1}^N\lambda_i^p\big\|(\uno \otimes a_i)x(i)(\uno \otimes a_i)\big\|_{S_p(\ell_2\otimes \ell_2^{k_i})}^p
\end{align}for every positive element $x\in S_p\otimes \ell_\infty^N(M_{k_i})$.

%%%%%%%%%
Indeed, in this case one can consider the set 
\begin{align*}
S=B^+_{S_{2p}^{k_1}}\cup B^+_{S_{2p}^{k_2}}\cup\cdots\cup B^+_{S_{2p}^{k_N}},
\end{align*}where $B^+_{S_{2p}^{k_i}}$ denotes the set of positive elements in the unit ball of $S_{2p}^{k_i}$.  We also define the set of functions
\begin{align*}
\mathcal F:=\Big\{g_{x_1,\cdots, x_m}:S\rightarrow \R \text{   } \big| \text{   }m\in \N, x_j\in S_p\otimes \ell_\infty^N(M_{k_i})  \text{ positive  for every } j=1,\cdots ,m\Big\},
\end{align*}where for every $a\in B^+_{S_{2p}^{k_i}}$
\begin{align*}
g_{x_1,\cdots, x_m}(a)=C^p\sum_{j=1}^m\Big(\big\|(\uno\otimes a)x_j(i)(\uno\otimes a)\big\|_{S_p(\ell_2\otimes \ell_2^{k_i})}^p- \big\|(id_{S_p}\otimes T)(x_j)\big\|_ {S_p[\B]}^p\Big)
\end{align*}for every $x_1,\cdots ,x_m$ positive elements in $S_p\otimes \ell_\infty^N(M_{k_i})$. Then, one can mimic the argument in \cite{Pisierbook2} by using (\cite[Lemma 5.2]{Pisierbook2}), (\cite[Lemma 1.1.4]{Pisierbook2}) and the fact that $S$ is compact in our case to deduce the existence of a probability distribution $(\beta_i)_{i=1}^N$ and a sequence of elements $(a_i)_{i=1}^N$ with $a_i\in  B^+_{S_{2p}^{k_i}}$ verifying
\begin{align}\label{majoration}
0\leq C^p\sum_{i=1}^N\sum_{j=1}^m\beta_i^\alpha\Big(\big\|(\uno \otimes a_i)x_j(i)(\uno \otimes a_i)\big\|_{S_p(\ell_2\otimes \ell_2^{k_i})}^p- \sum_{j=1}^m\big\|(id_{S_p}\otimes T)(x_j)\big\|_ {S_p[\B]}^p\Big)
\end{align}for every positive elements $x_1,\cdots ,x_m$ in $S_p\otimes \ell_\infty^N(M_{k_i})$. Then, (\ref{Pietsch positive elements}) can be obtained from (\ref{majoration}) by restricting to $m=1$ and defining $\lambda_i=\beta_i^\frac{1}{p}$ for every $i$.

Let us simplify notation by writing $\tilde{T}=id_{S_p}\otimes T$, $\tilde{a}_i=\uno\otimes a_i\in B(\ell_2\otimes \ell_2^{k_i})$ and $A=\sum_{i=1}^N\lambda_i^\frac{1}{2}e_i\otimes \tilde{a}_i\in \ell_\infty^N\otimes B(\ell_2\otimes \ell_2^{k_i})$. Then, Equation (\ref{Pietsch positive elements}) becomes
\begin{align}\label{Pietsch positive elements II}
\|\tilde{T}(x)\|_ {S_p[\B]}^p\leq C^p\|AxA\|_{\ell_p^N\big(S_p(\ell_2\otimes \ell_2^{k_i})\big)}^p
\end{align}for every positive element $x\in S_p\otimes \ell_\infty^N(M_{k_i})$. To finish the proof we will show that (\ref{Pietsch positive elements II}) holds for every $x\in S_p\otimes \ell_\infty^N(M_{k_i})$.

First note that we can assume that $A$ is invertible. To see this, just note that every element can be written as $x=pxp+(1-p)xp+px(1-p)+(1-p)x(1-p)$, where here $p$ denotes the support projection of $A$. Then, it is very easy to deduce from Proposition \ref{CS CP} and (\ref{Pietsch positive elements II}) that $\tilde{T}(x)=\tilde{T}(pxp)$. Therefore, if $A$ is not invertible, we can restrict to the finite dimensional von Neumann algebra $p\mathcal Ap$. On the other hand, for every $x\in \ell_\infty^N\big(B(\ell_2\otimes \ell_2^{k_i})\big)$ we have
\begin{align}\label{decomposition}
\|AxA\|_{\ell_p^N\big(S_p(\ell_2\otimes \ell_2^{k_i})\big)}=\inf\Big\{\|yA\|_{\ell_{2p}^N\big(S_{2p}(\ell_2\otimes \ell_2^{k_i})\big)}\|zA\|_{\ell_{2p}^N\big(S_{2p}(\ell_2\otimes \ell_2^{k_i})\big)}\Big\},
\end{align}where the infimum is taken over all possible decompositions $x=y^*z$. Indeed, inequality $\leq$ follows from the noncommutative H\"older's inequality. To see that the infimum is attained we use the fact that we can write $AxA=x_1x_2$ so that
\begin{align*}
\|AxA\|_{\ell_p^N\big(S_p(\ell_2\otimes \ell_2^{k_i})\big)}=\|x_1\|_{\ell_{2p}^N\big(S_{2p}(\ell_2\otimes \ell_2^{k_i})\big)}\|x_2\|_{\ell_{2p}^N\big(S_{2p}(\ell_2\otimes \ell_2^{k_i})\big)}.
\end{align*}Therefore, if we define $y^*=A^{-1}x_1$ and $z=x_2A^{-1}$, we have $y^*z=x$ and
\begin{align*}
\|yA\|_{\ell_{2p}^N\big(S_{2p}(\ell_2\otimes \ell_2^{k_i})\big)}\|zA\|_{\ell_{2p}^N\big(S_{2p}(\ell_2\otimes \ell_2^{k_i})\big)}=\|AxA\|_{\ell_p^N\big(S_p(\ell_2\otimes \ell_2^{k_i})\big)}.
\end{align*}
Thus, we finish our proof in the following way. Given any element $x\in S_p\otimes \ell_\infty^N(M_{k_i})$, we find elements $y,z$ so that $x=y^*z$ and
\begin{align*}
\|AxA\|_{\ell_p^N\big(S_p(\ell_2\otimes \ell_2^{k_i})\big)}=\|yA\|_{\ell_{2p}^N\big(S_{2p}(\ell_2\otimes \ell_2^{k_i})\big)}\|zA\|_{\ell_{2p}^N\big(S_{2p}(\ell_2\otimes \ell_2^{k_i})\big)}.
\end{align*}Then, according to Proposition \ref{CS CP} and (\ref{Pietsch positive elements II}) we have
\begin{align*}
\|\tilde{T}(x)\|_{S_p[\B]}& \leq \|\tilde{T}(y^*y)\|_{S_p[\B]}^\frac{1}{2}\|\tilde{T}(z^*z)\|_{S_p[\B]}^\frac{1}{2}\\& \leq C\|yA\|_{\ell_{2p}^N\big(S_{2p}(\ell_2\otimes \ell_2^k)\big)}\|zA\|_{\ell_{2p}^N\big(S_{2p}(\ell_2\otimes \ell_2^{k_i})\big)}\\&=C\|AxA\|_{\ell_p^N\big(S_p(\ell_2\otimes \ell_2^{k_i}\big)}.
\end{align*}
This proves c).

The final statement on the constant $C$ follows easily by standard arguments.
\end{proof}
\begin{remark}\label{General Proposition 2.1}In fact, a more involved argument based on \cite{JuRu} allows to prove Proposition \ref{CS CP} for completely positive and normal maps $T:\A \rightarrow \B$ between general von Neumann algebras. Hence, Theorem \ref{factorization-positivity} (and Theorem  \ref{factorization-positivity d-version} below) can be proved exactly in the same way for a general von Neumann algebra $\B$.
\end{remark}
The following corollary will be very important.
\begin{corollary}\label{cor-Theorm-fact}
Given a completely positive map $T:M_n\rightarrow \ell_\infty^N(M_k)$, we have that
\begin{align}\label{trace duality positive}
\pi_q^o(T)=\sup\Big\{\big|tr (S\circ T)\big|: \pi_{p}^o\big(S:\ell_\infty^N(M_k)\rightarrow M_n\big)\leq 1 \text{ with $S$ completely positive}  \Big\},
\end{align}where $\frac{1}{p}+\frac{1}{q}=1$.

In particular,
\begin{align}\label{key positive}
\pi_q^o(T)=\Big\|\emph{flip}\circ (id_{\ell_p^N(S_p^k)}\otimes T^*):\ell_p^N(S_p^k)\big[\ell_1^N(S_1^k)\big]\rightarrow S_1^n\big[\ell_p^N(S_p^k)\big]\Big\|_+
\end{align}and the norm can be computed restricting to elements of the form $\rho=\sum_{i=1}^N\lambda_i e_i\otimes e_i\otimes M_{a_i}$ with $a_i\in S_{2p}^k$ positive for every $i=1,\cdots ,N$. Here, we identify the tensor $M_{a_i}\in M_k\otimes M_k$ with the associated map $M_{a_i}:M_k\rightarrow M_k$.
\end{corollary}
\begin{proof}
According to (\ref{duality p-summing}) for any linear map $T:M_n\rightarrow \ell_\infty^N(M_k)$ we have
\begin{align}\label{trace duality}
\pi_q^o(T)=\sup\Big\{\big|tr (T\circ S)\big|: \pi_{p}^o\big(S:\ell_\infty^N(M_k)\rightarrow M_n\big)\leq 1 \Big\},
\end{align}where $\frac{1}{p}+\frac{1}{q}=1$. Therefore, we only have to prove inequality $\leq $ in (\ref{trace duality positive}).
Let us briefly explain the relation between the duality in (\ref{trace duality positive}) and Theorem  \ref{factorization-positivity}. We consider the diagram
$$\xymatrix@R=0.5cm@C=1cm {{\ell_\infty^N(M_k)}\ar[rd]^{\alpha}\ar[rr]^{S} & {} & 
{M_n}\ar[rr]^{T} & {} & {\ell_\infty^N(M_k).}\\
{ } & {S_p}\ar[ru]^{\beta} &{} &{} &{}}
$$We see that $tr (T\circ S)= \langle(id_{S_q}\otimes T)(\hat{\beta}), \hat{\alpha}\rangle$, where $\hat{\alpha}\in S_p(\ell_1^N(S_1^k))$ and $\hat{\beta}\in S_q\otimes_{min} M_n$ are the corresponding tensors to the maps $\alpha$ and $\beta$ respectively. Now, according to Theorem  \ref{factorization-positivity}, in order to norm $T$ we can assume that $\hat{\beta}$ is a positive element. This is equivalent to say that the map $\beta$ can be assumed to be completely positive (see for instance \cite[Theorem 3.14]{Paulsen}). Hence, it suffices to show that $\alpha$ can be assumed to be a completely positive map. To this end, we will show that the element $\hat{\alpha}$ can be assumed to be positive. Since $\pi_{p}^o\big(S:\ell_\infty^N(M_k)\rightarrow M_n\big)\leq 1$, we can assume that $\hat{\alpha}$ and $\hat{\beta}$ are in the unit ball of  $S_p(\ell_1^N(S_1^k))$ and $S_q\otimes_{min} M_n$ respectively. In particular, according to (\ref{Norm in S_p[E]}) there must exist elements $A,B$ in the unit ball of $S_{2p}$ and $\hat{Y}$ in the unit ball of $B(\ell_2)\otimes_{min} \ell_1^N(S_1^k)$ such that $$\hat{\alpha}=(A\otimes \uno_{\ell_\infty^N(M_k)})\hat{Y}(B\otimes \uno_{\ell_\infty^N(M_k)}).$$We claim that $\hat{Y}=\hat{Y}_1\hat{Y}_2$ with both $\hat{Y}_1\hat{Y}_1^*$ and $\hat{Y}_2^*\hat{Y}_2$ in the unit ball of $B(\ell_2)\otimes_{min} \ell_1^N(S_1^k)$. Once we have this, we can conclude the proof by using Cauchy-Schwartz inequality:
\begin{align*}
\big|tr (T\circ S)\big|= \big|\langle(id_{S_q}\otimes T)(\hat{\beta}), \hat{\alpha}\rangle\big|=\big|\langle(id_{S_q}\otimes T)(\hat{\beta}), Z_1Z_2\rangle\big|\\\leq \big|\langle(id_{S_q}\otimes T)(\hat{\beta}), Z_1Z_1^*\rangle\big|^\frac{1}{2}\big|\langle(id_{S_q}\otimes T)(\hat{\beta}), Z_2^*Z_2\rangle\big|^\frac{1}{2},
\end{align*}where we denote $Z_1=(A\otimes \uno_{\ell_\infty^N(M_k)})\hat{Y}_1$ and $Z_2=\hat{Y}_2(B\otimes \uno_{\ell_\infty^N(M_k)})$. Here, we have used that $(id_{S_q}\otimes T)(\hat{\beta})$ is a positive element, since $\hat{\beta}$ is positive and $T$ is completely positive. Using that both  $Z_1Z_1^*$ and $Z_2^*Z_2$ are positive elements in the unit ball of $S_p(\ell_1^N(S_1^k))$, we conclude that 
$$\big|tr (T\circ S)\big|\leq \sup\Big\{\big| \langle(id_{S_q}\otimes T)(\hat{\beta}),\hat{\alpha}\rangle\big|\Big\},$$where the supremum runs over al positive elements $\hat{\alpha}$ in the unit ball of $S_p(\ell_1^N(S_1^k))$, as we wanted.

It remains to prove our claim. To this end,  we recall that
$$\big\|\hat{Y}\big\|_{B(\ell_2)\otimes_{min} \ell_1^N(S_1^k)}=\big\|Y:\ell_\infty^N(M_k)\rightarrow B(\ell_2)\big\|_{cb},$$where $Y$ is the linear map associated to the tensor $\hat{Y}.$ Then, we can invoke Wittstock's factorization theorem (see \cite[Theorem 8.4]{Paulsen}) to find a Hilbert space $K$, a $*$-homomorphism $\pi:\ell_\infty^N(M_k)\rightarrow B(K)$ and some contractions $V,W:\ell_2\rightarrow K$ so that
$$Y(x)=V^*\pi(x)W$$for every $x\in \ell_\infty^N(M_k)$. By regarding $\pi(e_{i,j}^k)=\pi(e_{i,1}^ke_{1,j}^k)=\pi(e_{i,1}^k)\pi(e_{1,j}^k)$, we can define linear maps $Y_1: \ell_\infty^N(M_k)\rightarrow B(K, \ell_2)$, and $Y_2:\ell_\infty^N(M_k)\rightarrow B(\ell_2,K)$
so that $\hat{Y}_1\hat{Y}_2=\hat{Y}$. Furthermore, the linear maps associated to $\hat{Y}_1\hat{Y}_1^*$ and $\hat{Y}_2^*\hat{Y}_2$ are given by $Y_1Y_1^*(e_{i,j}^k)=V^*\pi(e_{i,j}^k)V$ and $Y_2^*Y_2(e_{i,j}^k)=W^*\pi(e_{i,j}^k)W$ for every $i,j,k$ respectively. Thus, we easily conclude that both elements are in the unit ball of the space $B(\ell_2)\otimes_{min} \ell_1^N(S_1^k)$.

Finally, the last part of the statement (\ref{key positive}) follows directly from Theorem \ref{factorization-positivity}, (\ref{trace duality positive}) and duality.
Indeed, on the one hand, the dual version of (\ref{Def completely p-summing}) in our particular case says that
\begin{align*}
\pi_q^o(T):=\big\|\text{flip}\circ (id_{S_p}\otimes T^*):S_p[\ell_1^N(S_1^k)]\rightarrow S_1^n[S_p]\big\|.
\end{align*}On the other hand, once we know that $S$ can be assumed to be positive, Theorem  \ref{factorization-positivity} tells us that the real picture corresponding to Equation (\ref{trace duality positive}) is 
$$\xymatrix@R=0.5cm@C=1cm {{\ell_\infty^N(M_k)}\ar[rd]^{\alpha}\ar[rr]^{S} & {} & 
{M_n}\ar[rr]^{T} & {} & {\ell_\infty^N(M_k)}\\
{ } & {\ell_p^N(S_p^k)}\ar[ru]^{\beta} &{} &{} &{}},
$$where $\hat{\alpha}=\sum_{i=1}^N\lambda_ie_i\otimes e_i\otimes M_{a_i}\in \ell_p^N(S_p^k)\big(\ell_1^N(S_1^k)\big)$. This means, in particular, that $S_p$ can be replaced by $\ell_p^N(S_p^k)$ in the previous expression for $\pi_q^o(T)$ and, furthermore, the norm is attained on elements which have the same form as $\hat{\alpha}$. Indeed, by considering optimal maps $\alpha$ and $\beta$ we have 
\begin{align*}
\big|tr (T\circ S)\big|\leq \big| \langle(id_{\ell_p^N(S_p^k)}\otimes T)(\hat{\beta}),\hat{\alpha}\rangle\big|=\big| \langle \hat{\beta},(id_{\ell_p^N(S_p^k)}\otimes T^*)(\hat{\alpha})\rangle\big|.
\end{align*}
\end{proof}
In this work we will need to consider the following generalization of completely $p$-summing maps. A linear map between operator spaces $T:E\rightarrow F$ is called \emph{$\ell_p(S_p^d)$-summing map} if
\begin{align*}
\pi_{p,d}(T):=\big\|id_{\ell_p(S_p^d)}\otimes T:\ell_p(S_p^d)\otimes_{min} E\rightarrow \ell_p(S_p^d)[F]\big\|< \infty.
\end{align*}
The author should note the difference between notation $\pi_{p,d}(T)$ above and notation $\pi_p^d(T)$ introduced in page 9-10.

Note that the case $d=\infty$ above corresponds to the completely $p$-summing maps. On the other hand, the case $d=1$ was introduced by the first author and they are called \emph{$(p,cb)$-summing maps} (see \cite{Junge-Hab}). They can be considered as a generalization of the absolutely $p$-summing maps to an intermediate setting between these maps and the completely $p$-summing maps. It can be seen that $\ell_p(S_p^d)$-summing maps verify a Pietsch factorization theorem analogous to the theorem for completely $p$-summing maps (see \cite[remark 5.11]{Pisierbook2}). Actually, following the proof of Theorem \ref{factorization-positivity} word by word one can show an analogous version of the theorem for the $\ell_p(S_p^d)$-summing maps. More specifically, one has the following result.
\begin{theorem}\label{factorization-positivity d-version}
Let $\A$ and $\B$ be finite dimensional  von Neumann algebras and let $T:\A\rightarrow \B$ be a completely positive map. Then, the following assertions are equivalent:
\begin{enumerate}
\item[a)] $\pi_{p,d}(T)\leq C$.
\item[b)] $\big\|id_{\ell_p(S_p^d)}\otimes T:\ell_p(S_p^d)\otimes_{min} \A\rightarrow \ell_p(S_p^d)[\B]\big\|_+\leq C$.
\item[c)]  There exists a positive element $a\in \A$ verifying $\|a\|_{L_{2p}(\A)}\leq1$ such that for every $x\in S_p^d\otimes_{min} \A$ we have $$\big\|\big(id_{S_p^d}\otimes T\big)(x)\big\|_{S_p^d[\B]}\leq C\big\|(\uno_{S_p^d}\otimes a)x(\uno_{S_p^d}\otimes a)\big\|_{S_p^d[L_p(\A)]}.$$
\item[d)] There exist a positive element $a\in \A$ verifying $\|a\|_{L_{2p}(\A)}\leq1$ and a completely positive linear map $\alpha:L_p(\A)\rightarrow \B$ such that $T=\alpha\circ M_{a}$ and $\|\alpha\|_d\leq C$.
\noindent Here $M_{a}:\A\rightarrow L_p(\A)$ is the linear map defined by $M_a(x)=axa$ for every $x\in \A$ and $\|\alpha\|_d=\big\|id_d\otimes \alpha:M_d(L_p(\A))\rightarrow M_d(\B)\big\|$.
\end{enumerate}
Furthermore, $\pi_{p,d}(T)=\inf\Big\{C: C \text{   verifies any of the above conditions}\Big\}.$
\end{theorem}
In particular, in order to compute the $\pi_{p,d}$ norm of a completely positive map $T$ between finite dimensional von Neumann algebras, it suffices to consider positive elements in the definition of the norm. We will consider the following point of view.
\begin{remark}\label{positive embedding}
Given $M_n$, for every natural number $1\leq d\leq n$ we consider the unit ball of the space of completely bounded maps from $M_n$ to $M_d$, $\mathcal K=B_{CB(M_n, M_d)}$. Then, we can define the map $$j_d:M_n\rightarrow \ell_\infty(\mathcal K, M_d)$$ given by $$j_d(A)=\big(\phi(A)\big)_{\phi\in \mathcal K} \text{   } \text{   for every    }  \text{   }  A\in M_n.$$
It is not difficult to see that the previous map $j_d$ defines a $d$-isometry. That is, the map $$id_d\otimes j_d:M_d\otimes_{min} M_n\rightarrow M_d\otimes_{min}\ell_\infty(\mathcal K, M_d)$$is an isometry. Therefore, it follows from the very definition of the $\pi_{p,d}$ norm that for a map $\mathcal J:M_n\rightarrow M_n$ we have
\begin{align}\label{d-isometry}
\pi_{p,d}(\mathcal J)=\pi_{p,d}(j_d\circ \mathcal J)=\pi_p^o(j_d\circ \mathcal J),
\end{align}where the last equality comes from the Pietch factorization theorem and the fact that a map $T:X\rightarrow \ell_\infty(\mathcal K, M_d)$ verifies that $\|T\|_{cb}=\|T\|_d$.

When $\mathcal T:M_n\rightarrow M_n$ is completely positive we can actually consider the set
\begin{align}\label{set P}
\mathcal P=CPU(M_n, M_d):=\Big\{T:M_n\rightarrow M_d, T \text{   } \text{  }  \text{is completely positive and unital}\Big \}
\end{align}to obtain an equation similar to (\ref{d-isometry}). Indeed, if we consider the map $j_d:M_n\rightarrow \ell_\infty(\mathcal P, M_d)$, where $j_d$ is defined as above,$$j_d(A)=\big(\phi(A)\big)_{\phi\in \mathcal P} \text{   } \text{   for every    }  \text{   }  A\in M_n,$$
one can easily check that this map is a $d$-isometry on positive elements (that is, the map $id_d\otimes j_d$ is an isometry when acting on positive elements in $M_d\otimes_{min} M_n$). Let us briefly explain this point. Given any positive element $A\in M_{dn}$, we must show that there exists a completely positive map $T:M_n\rightarrow M_d$ such that $\|A\|=\|(id_d\otimes T)(A)\|_{M_{d^2}}$. Now, since $A$ is positive, there must exist a unit element $\xi\in \ell_2^d\otimes \ell_2^n$ such that $\|A\|=\langle \xi, A\xi\rangle$. On the other hand, since $\xi$ has rank $d$, we can find a projection $P:\ell_2^n\rightarrow \ell_2^d$ so that
$\langle \xi, A\xi\rangle=\langle \eta, (id_{\ell_2^d}\otimes P)A(id_{\ell_2^d}\otimes P^*)\eta\rangle$ for a certain $\eta=(id_{\ell_2^d}\otimes P)\xi\in \ell_2^d\otimes \ell_2^d$. Therefore, the completely positive and unital map $T:M_n\rightarrow M_d$ defined by $x\mapsto PxP^*$ verifies what we want.

By the nice behavior of the $\ell_p(S_p^d)$-summing maps explained in Theorem \ref{factorization-positivity d-version} one can easily deduce that
\begin{align*}
\pi_{p,d}(\mathcal T)=\pi_{p,d}(j_d\circ \mathcal T)=\pi_p^o(j_d\circ \mathcal T).
\end{align*}
\end{remark}
We will need to give a more general definition of a quantum channel in order to consider also the case of infinite dimensional von Neumann algebras. Let $\mathcal H_1$ and $\mathcal H_2$ be two complex Hilbert spaces and let us denote by $\mathcal B(\mathcal H_i)$ the von Neumann algebra of all bounded operators from $\mathcal H_i$ to $\mathcal H_i$ for $i=1,2$. Let us also denote $S_1(\mathcal H_i)$ the space of trace class operators from $\mathcal H_i$ to $\mathcal H_i$ for $i=1,2$. We can define a quantum channel as a completely positive and trace preserving map $\mathcal N:S_1(\mathcal H_1)\rightarrow S_1(\mathcal H_2)$. In this case, we say that we are describing the channel in the \emph{Schr\"odinger picture}. On the other hand, for a given quantum channel we can consider the dual map to obtain a completely positive map $\mathcal \Ne^*:\mathcal B(\mathcal H_2)\rightarrow \mathcal B(\mathcal H_1)$ which turns out to be unital. In this case, we say that we are working with the \emph{Heisenberg picture} of the channel\footnote{One could define a quantum channel as a completely positive and unit map (Heisenberg picture) and, then, consider the dual map of it to define the Schr\"odinger picture. This gives a more general definition of a quantum channel, where the states are just (normalized) positive functionals on our von Neumann algebra. However, these considerations will not be relevant in this work.}. Although in this work we are interested in quantum channels defined on finite dimensional von Neumann algebras, in order to study their capacities we will need to consider certain channels defined on the direct sum of infinitely many copies of finite dimensional matrix algebras $\bigoplus_{i=1}^\infty M_{n_i}$. More precisely, we will consider channels of the form $\Ne:\ell_1(I,S_1^d)\rightarrow S_1^n$ so that the adjoint is defined on $\Ne^*:M_n\rightarrow \ell_\infty(I,M_d)$. Here, $I$ denotes an arbitrary index set. Since most of our results can be stated for general von Neumann algebras\footnote{To avoid technical issues, we will restrict to hyperfinite von Neumann algebras, where all our elements are perfectly well defined.} $\mathcal A$ and $\mathcal B$, we will denote a general quantum channel by $\Ne:L_1(\A)\rightarrow L_1(\B)$ or $\Ne^*:\B\rightarrow \A$.

The following definition will be crucial in the rest of the work.
\begin{definition}
Given two hyperfinite von Neumann algebras $\A$, $\B$ and a quantum channel $\Ne:L_1(\A)\rightarrow L_1(\B)$, we define
\begin{align*}
\tilde {C}(\Ne)=\lim_{q\rightarrow \infty}q\ln \big(\Pi_q^o(\Ne^*)\big).
\end{align*}
\end{definition}
The following two lemmas tell us about the soundness of the previous definition.
\begin{lemma}\label{decreasing}
Let $\A$ and $\B$ be two hyperfinite von Neumann algebras and let $T:\A\rightarrow \B$ be a linear map such that $\|T\|_{cb}=1$. The function $f_T:[1,\infty)\rightarrow \R$ defined by $f_T(q)=q\ln \pi_q^o(T)$ is non-negative and non-increasing. In particular, $$\lim_{q\rightarrow \infty}q\ln \pi_q^o(\Ne^*)=\inf_qq\ln \pi_q^o(\Ne^*)\leq \pi_1^o(\Ne^*)< \infty$$ is well defined for every quantum channel $\Ne:\ell_1(I,S_1^d)\rightarrow S_1^n$.
\end{lemma}
\begin{proof}
On the one hand, the non negativity follows from the fact that $\pi_q^o(T)\geq \|T\|_{cb}=1$ for every $q\geq 1$. In order to prove the second assertion let us consider $1\leq q_1\leq q_2< \infty$. Then, by using a standard interpolation argument one can show that for every $\alpha\in (0,1)$ such that $\frac{1}{q_2}=\frac{\alpha}{q_1}+\frac{1-\alpha}{\infty}$ we have $\pi_{q_2}^o(T)\leq \pi_{q_1}^o(T)^\alpha\pi_{\infty}^o(T)^{1-\alpha}= \pi_{q_1}^o(T)^\alpha.$ Indeed, this follows easily by just recalling that 
$$\pi_{q}^o(T)=\sup_k\big\|id_{S_p^k}\otimes T: S_q^k\otimes_{min}\mathcal A\rightarrow S_q^k[\mathcal A]\big\|$$for every $1\leq q\leq \infty$ and the fact that the spaces  $S_q^k\otimes_{min}\mathcal A$ and $S_q^k[\mathcal A]$ interpolates properly in $q$ (see \cite[Chapter 1]{Pisierbook2}). Therefore,
\begin{align*}
q_2\ln \pi_{q_2}^o(T)=\ln \big(\pi_{q_2}^o(T)^{q_2}\big)\leq \ln \big(\pi_{q_1}^o(T)^{q_2\alpha}\big)=\ln \big(\pi_{q_1}^o(T)^{q_1}\big)=q_1\ln \pi_{q_1}^o(T).
\end{align*}
The last assertion follows immediately by the statement of the lemma and the fact that $\Ne^*:M_n\rightarrow \ell_\infty(I,M_d)$ has finite rank.
\end{proof}
\begin{lemma}\label{quotient}
Let $f_T:[1,\infty)\rightarrow \R^+$ be a function such that $\lim_{q\rightarrow \infty}f(q)=1$. Then, $\lim_{q\rightarrow \infty}q\ln f(q)$ exists if and only if $\frac{d}{dp}[f(q)]|_{p=1}:=\lim_{p\rightarrow 1^+}\frac{f(q)-1}{p-1}$ exists and in this case the limits are the same. Here, $\frac{1}{p}+\frac{1}{q}=1$.

In particular, for any quantum channel $\Ne:\ell_1(I,S_1^d)\rightarrow S_1^n$ we have
\begin{align*}
\tilde {C}(\Ne)=\frac{d}{dp}[\pi_q^o(\Ne^*)]|_{p=1}.
\end{align*}
\end{lemma}
\begin{proof}
Since $\lim_{q\rightarrow \infty}f(q)=1$, we have $\lim_{q\rightarrow \infty}\frac{\ln(f(q))}{f(q)-1}=1$. Then, the first assertion of the statement follows easily from the fact that $\frac{f(q)-1}{p-1}=(q-1)(f(q)-1)$ and, therefore, $\lim_{p\rightarrow 1^+}\frac{f(q)-1}{p-1}=\lim_{q\rightarrow \infty}(q-1)(f(q)-1)$. Here, we have used that $\frac{1}{p}+\frac{1}{q}=1$.

On the other hand, since $\Ne^*$ is a completely positive and unital map between von Neumann algebras, we know that $\|\Ne^*\|_{cb}=1$. Then, by denoting $f(q)=\pi_q^o(\Ne^*)$ we have that $\lim_{q\rightarrow \infty}f(q)=\lim_{q\rightarrow \infty}\pi_q^o(\Ne^*)= \|\Ne^*\|_{cb}=1$. Therefore, the second assertion of the lemma is a direct consequence of Lemma \ref{decreasing}.
\end{proof}
%
%%%%%%%%%%%%%%%%%%%%%%%%%%%%%%%%%%%%%%%%%%%
%
%
%
%
\section{Main Theorem: Restricted capacities via $\ell_p(S_p^d)$-summing maps}\label{main results section}
In this section we will prove Theorem \ref{mainI}. In fact, this result will follow from a more general theorem:
\begin{theorem}\label{mainII}
Let $\big(\Ne_i:S_1^d\rightarrow S_1^n\big)_{i\in I}$ be a family of quantum channels indexed in a set $I$ and let $d$ be any natural number verifying $1\leq d\leq n$. Let us define the quantum channel $$\Ne:\ell_1(I,S_1^d)\rightarrow S_1^n$$ by linearity with $$\Ne(e_i\otimes \rho_i)=\Ne_i(\rho_i)  \text{   }  \text{    for every     } \text{   } i\in I.$$ Then,
\begin{align*}
\tilde{C}(\Ne)=C_E((\Ne_i)_i):=\sup\Big\{S\Big(\sum_{i=1}^N \lambda_i(tr_d\otimes id_n)\big((id_d\otimes \Ne_i)(\eta_i)\big)\Big)\\ +\sum_{i=1}^N \lambda_i\Big[S\Big((id_d\otimes tr_n)\big((id_d\otimes \Ne_i)(\eta_i)\big)\Big)-S\Big((id_d\otimes \Ne_i)(\eta_i)\Big) \Big] \Big\},
\end{align*} where the supremum runs over all $N\in \N$, all probability distributions $(\lambda_i)_{i=1}^N$ and all pure states $\eta_i\in S_1^d\otimes S_1^d$.
\end{theorem}
\begin{remark}
Using that $\Ne_i$ is a quantum channel for every $i\in I$ it is trivial to see that the right hand side term in Theorem \ref{mainII} can be written as
\begin{align}\label{C_E good}
\sup\Big\{S\Big(\sum_{i=1}^N \lambda_i\Ne_i\big((tr_d\otimes id_d)(\eta_i)\big)\Big) +\sum_{i=1}^N \lambda_i\Big[S\Big((id_d\otimes tr_d)(\eta_i)\Big)-S\Big((id_d\otimes \Ne_i)(\eta_i)\Big)  \Big]\Big\},
\end{align}where the supremum runs over all $N\in \N$, all probability distributions $(\lambda_i)_{i=1}^N$ and all pure states $\eta_i\in S_1^d\otimes S_1^d$.
\end{remark}
We will first show how to obtain Theorem \ref{mainI} from Theorem \ref{mainII}.

Let $\Ne:S_1^n\rightarrow S_1^n$ be a quantum channel. Then, we define the new quantum channel
\begin{align*}
\hat{\Ne}:\ell_1(\mathcal P, S_1^d)\rightarrow S_1^n,
\end{align*}given by $\hat{\Ne}\big((\rho_\sigma)_{\sigma\in \mathcal P}\big)=\sum_{\sigma\in \mathcal P}\Ne\big(\sigma^*(\rho_\sigma)\big)$, where $\mathcal P$ is defined as in (\ref{set P}). That is, $\hat{\Ne}$ is defined by the family of channels $\big(\Ne \circ \sigma^*:S_1^d\rightarrow S_1^n\big)_{\sigma\in \mathcal P}.$ On the other hand, it is very easy to check that $\hat{\Ne}^*=j_d\circ \Ne^*:M_n\rightarrow \ell_\infty(\mathcal P, M_d)$. According to Remark \ref{positive embedding}, we conclude that $\pi_{q,d}(\mathcal \Ne^*)=\pi_{q,d}(\hat{\Ne}^*)=\pi_q^o(\hat{\Ne}^*).$ Therefore, Theorem \ref{mainII} tells us that
\begin{align*}
\lim_{q\rightarrow \infty}q\ln \pi_{q,d}(\Ne^*)=\lim_{q\rightarrow \infty}q\ln \pi_{q}^o(\hat{\Ne}^*)=\sup\Big\{S\Big(\sum_{i=1}^N \lambda_i(\Ne\circ \phi_i)\big((tr_d\otimes id_d)(\eta_i)\big)\Big)
\\+\sum_{i=1}^N \lambda_i\Big[S\Big((id_d\otimes tr_d)(\eta_i)\Big)-S\Big(\big(id_d\otimes (\Ne\circ \phi_i)\big)(\eta_i)\Big) \Big]\Big\}.
\end{align*}Here, the supremum runs over all $N\in \N$, all probability distributions $(\lambda_i)_{i=1}^N$, all pure states $\eta_i\in S_1^d\otimes S_1^d$ and all quantum channels $\phi_i:S_1^d\rightarrow S_1^n$. So we obtain Theorem \ref{mainI}. Note that we have used that $\sigma:M_n\rightarrow M_d$ is completely positive and unital if and only if $\phi=\sigma^*:S_1^d\rightarrow S_1^n$ is completely positive and trace preserving.

The rest of the section will be devoted to proving Theorem \ref{mainII}. For the first inequality, $\tilde{C}(\Ne)\leq C_E((\Ne_i)_i)$, we will need the following well known lemma.
\begin{lemma}[\cite{KiKo}]\label{GenralLieb}
For every positive element in $M_n\otimes M_m$ we have that
\begin{align*}
\|x\|_{S_1^n[S_p^m]}\leq \Big\|\big((id_n\otimes tr_m)(x^p)\big)^\frac{1}{p}\Big\|_{S_1^n}.
\end{align*}
\end{lemma}
We will also use the following two well known results about the von Neumann entropy. The first one is about its continuity and the second one relates the von Neumann entropy of a state with its $p$-norm.
\begin{theorem}\cite[Theorem 1]{Aud1}\label{cont-Entropy}
For all $n$-dimensional states $\rho$, $\sigma$ we have
\begin{align*}
|S(\rho)-S(\sigma)|\leq T\ln (n-1)+ H((T,1-T)),
\end{align*}where $T=\frac{\|\rho-\sigma\|_1}{2}$ and $H$ denotes the Shannon entropy.

In particular, given $\epsilon> 0$ there exists a $\gamma=\gamma(\epsilon,n)> 0$ such that for every positive operators $\rho$ and $\sigma$ in $M_n$ so that $\|\rho-\sigma\|_1\leq\gamma$, $\big|\|\rho\|_1-1\big|< \gamma$ and $\big|\|\sigma\|_1-1\big|< \gamma$ we have $$|S(\rho)-S(\sigma)|\leq \epsilon.$$
\end{theorem}
Indeed, the second part of the statement can be obtained by writing 
\begin{align*}
|S(\rho)-S(\sigma)|\leq \Big|S(\rho)-S\big(\frac{\rho}{\|\rho\|_1}\big)\Big|+\Big|S\big(\frac{\rho}{\|\rho\|_1}\big)-S\big(\frac{\sigma}{\|\sigma\|_1}\big)\Big|+\Big|S\big(\frac{\sigma}{\|\sigma\|_1}\big)-S(\sigma)\Big|.
\end{align*}Then, the first and the third of these terms can be easily controlled by considering the eigenvalues $(\lambda_i)_{i=1}^n$ of $\rho$ and $\sigma$ respectively while the second term is controlled by the first part of the statement.
\begin{theorem}\label{derivate-Entropy}
The function $F(\rho,p)= \frac{1-\|\rho\|_p}{p-1}$ is well defined for $p$ positive with $p\neq 1$ and $\rho$ a density matrix. It can be extended by continuity to $p\in (0,\infty)$ and this extension verifies $$F(\rho,1)=-\frac{d}{dp}\|\rho\|_p\big|_{p=1}=S(\rho).$$Moreover, the convergence at $p=1$ is uniform in the states $\rho$.

In particular, for every net $(\rho_p)_p$ of states such that  $\lim_{p  \rightarrow 1}\rho_p= \rho$ in the trace class norm, we have that $\lim_{p  \rightarrow 1}F(\rho_p,p)=S(\rho)$.
\end{theorem}
Indeed, although the first part of Theorem \ref{derivate-Entropy} was proved in  \cite{AHW} for the function  $\frac{1-\|\rho\|_p^p}{p-1}$, it is very easy to conclude that, then, the same result must hold for 
the function $F(\rho,p)$. On the other hand, the second part of the statement is a direct consequence of the uniform convergence and the continuity of the von Neuman entropy stated in Theorem \ref{cont-Entropy}.
The following two remarks will be very useful:
\begin{remark}\label{aprox log}
For any real numbers $\lambda\in (0,1]$ and $p\geq 1$ we have $\lambda- \lambda^p=\int_1^p\lambda^q (-\ln \lambda) dq$. Therefore,
\begin{align*}
\lambda^p(-\ln \lambda)(p-1)\leq \lambda- \lambda^p\leq \lambda(-\ln \lambda)(p-1).
\end{align*}Taking $\mu=\lambda^p\in (0,1]$ we obtain
\begin{align*}
\mu(- \ln \mu^\frac{1}{p})(p-1)\leq\mu^\frac{1}{p}-\mu\leq \mu^\frac{1}{p}(- \ln \mu^\frac{1}{p}) (p-1).
\end{align*}
\end{remark}
\begin{remark}\label{Remark-tilde(C)}
We will restrict our study to those quantum channels of the form $\Ne:\ell_1^N(S_1^d)\rightarrow S_1^n$, where $\Ne$ is defined by a family of quantum channels $\big(\Ne_i:S_1^d\rightarrow S_1^n\big)_{i=1}^N$ such that $\Ne(\sum_{i=1}^Ne_i\otimes \rho_i)=\sum_{i=1}^N\Ne_i(\rho_i)$. In this case, according to Corollary \ref{cor-Theorm-fact} and Lemma \ref{quotient} we can write
\begin{align}
\tilde {C}(\Ne)=\lim_{p\rightarrow 1^+}\frac{1}{p-1}\Big(\Big\|\text{flip}\circ (id_{\ell_\infty^N(M_d)}\otimes \Ne):\ell_p^N(S_p^d)\big[\ell_1^N(S_1^d)\big]\rightarrow S_1^n\big[\ell_p^N(S_p^d)\big]\Big\|_+-1\Big).
\end{align}Moreover, according to Corollary \ref{cor-Theorm-fact}, for a fixed $p$ we know that the previous norm is attained on a positive element of the form $\rho_p=\sum_{i=1}^N\lambda_i(p) e_i\otimes e_i\otimes M_{a_i}(p)$.
\end{remark}
\begin{proof}[Proof of inequality $\leq$ in Theorem \ref{mainII}]\label{firs-inequality}

Let $\epsilon> 0$. We must find a $\delta=\delta(\epsilon)> 0$ such that $p-1< \delta$ implies $$\frac{\pi_q^o(\Ne^*)-1}{p-1}\leq C_E((\Ne_i)_i)+\epsilon,$$where $\frac{1}{p}+\frac{1}{q}=1$. Using a compactness argument, for every fixed $p> 1$ we can find an $N=N(\epsilon, p)\in \N$ such that
\begin{align}\label{standard arguments}
\Big|\pi_q^o(\Ne^*)-\pi_q^o\big((\Ne|_{\ell_1^{N(\epsilon, p)}(S_1^d)})^*\big)\Big|< (p-1)\epsilon.
\end{align}
Therefore, it suffices to prove that
\begin{align*}
\frac{\pi_q^o(\Ne^*)-1}{p-1}\leq C_E((\Ne_i)_i) + \epsilon,
\end{align*}where we consider $\Ne:\ell_1^{N(\epsilon, p)}(S_1^d)\rightarrow S_1^n$.
According to Remark \ref{Remark-tilde(C)}, for a fixed $p> 1$ we have
\begin{align*}
\pi_q^o(\Ne^*)= \sup_{\rho_p}\frac{\big\|\text{flip}\circ (id_{\ell_\infty^{N(\epsilon, p)}(M_d)} \otimes \Ne) (\rho_p)\big\|_{S_1^n\big[\ell_p^{N(\epsilon, p)}(S_p^d)\big]}} {\|\rho_p\|_{\ell_p^{N(\epsilon, p)}(S_p^d)\big[\ell_1^{N(\epsilon, p)}(S_1^d)\big]}}
\end{align*}and this supremum is attained on a positive element of the form  $$\rho_p=\sum_{i=1}^{N(\epsilon, p)}\lambda_i(p)e_i\otimes e_i\otimes M_{a_i(p)}.$$Assuming that $\|\rho_p\|_{\ell_p^{N(\epsilon,p)^2}(S_1^{d^2})}=1$ (otherwise we can normalize) we can write
\begin{align}\label{norming element}
\rho_p=\sum_{i=1}^{N(\epsilon, p)}\beta_i(p)e_i\otimes e_i\otimes B_i(p)
\end{align} for certain positive numbers $(\beta_i(p))_{i=1}^N$ verifying
\begin{align}\label{normalization condition I}
\sum_{i=1}^{N(\epsilon, p)}\beta_i(p)^p=1
\end{align}and where
\begin{align}\label{normalization condition II}
\|B_i(p)\|_{S_1^{d^2}}=1
 \end{align}is the tensor associated to an operator of the form $M_{b_i(p)}$ for every $i=1,\cdots, N(\epsilon, p)$. In particular, $B_i(p)$ is a state in $S_1^d\otimes S_1^d$ for every $p,i$.

Now, we will choose our $\delta=\delta(\epsilon, n,d)>0$ ($n$ and $d$ are fixed parameters in the problem) independently from $p$, so that Equations (\ref{condition-p-1}), (\ref{claim first term}) and (\ref{claim second term}) are verified whenever $p-1< \delta$. Note that these equations depend on $N=N(\epsilon,p)$. However, the crucial point here is that this dependence does not play any role once we have our normalization conditions (\ref{normalization condition I}) and (\ref{normalization condition II}). This is what makes it possible to choose $\delta$ independent from $p$. Here, we will just explain how such a $\delta$ can be chosen and we will not make the computations to give an explicit one. 

For $p-1< \delta$, let us consider the corresponding element $\rho_p=\sum_{i=1}^{N(\epsilon, p)}\beta_i(p)e_i\otimes e_i\otimes B_i(p)$. From this point on we will remove the dependence of $p$ and $\epsilon$ from the notation of $N$, $\beta_i$ and $B_i$ for every $i=1,\cdots, N$.  We will denote $\tilde{B}_i=(id_d \otimes \Ne_i)(B_i)$, which is a state in $S_1^d\otimes S_1^n$, and $$\xi_p=\text{flip}\circ (id_{\ell_\infty^N(M_d)} \otimes \Ne) (\rho_p)=\sum_{i=1}^N\beta_ie_i\otimes \text{flip}(\tilde{B}_i).$$ Now, using this notation we can write $\frac{\pi_q^o(\Ne^*)-1}{p-1}$ as
\begin{align*}
\frac{1}{p-1}\Big(\frac{\|\xi_p\|_{S_1^n\big[\ell_p^N(S_p^d)\big]}} {\|\rho_p\|_{\ell_p^N(S_p^d)\big[\ell_1^N(S_1^d)\big]}}-1\Big)=\frac{1}{\|\rho_p\|_{\ell_p^N(S_p^d)\big[\ell_1^N(S_1^d)\big]}}
\Big(\frac{\|\xi_p\|_{S_1^n\big[\ell_p^N(S_p^d)\big]}-\|\rho_p\|_{\ell_p^N(S_p^d)\big[\ell_1^N(S_1^d)\big]}} {p-1}\Big),
\end{align*}which can be written as
\begin{align}\label{three terms}
\frac{1}{\|\rho_p\|_{\ell_p^N(S_p^d)\big[\ell_1^N(S_1^d)\big]}}\Big(\frac{\|\xi_p\|_{S_1^n\big[\ell_p^N(S_p^d)\big]}-
\|\xi_p\|^p_{\ell_p^N(S_p^{dn})}}{p-1}+ \frac{\|\xi_p\|^p_{\ell_p^N(S_p^{dn})}-1}{p-1}+ \frac{1-\|\rho_p\|_{\ell_p^N(S_p^d)\big[\ell_1^N(S_1^d)\big]}}{p-1}\Big).
\end{align}

Using Remark \ref{(p,1)}, (\ref{normalization condition I}) and (\ref{normalization condition II}) we see that we can find a $\delta=\delta(\epsilon, d)$ such that $p-1<\delta$ guarantees
\begin{align}\label{condition-p-1}
\frac{1}{\|\rho_p\|_{\ell_p^N(S_p^d)\big[\ell_1^N(S_1^d)\big]}}=\frac{1}{\big(\sum_{i=1}^N\beta_i^p\|(id_d\otimes tr_d)(B_i)\|_{S_p^d}^p\big)^\frac{1}{p}}\leq 1+\epsilon.
\end{align}Thus, we need to study the three terms in Equation (\ref{three terms}). 

We claim that we can find a $\delta=\delta(\epsilon, n,d)>0$ so that if $p-1<\delta$ we have
\begin{align}\label{claim first term}
\frac{\|\xi_p\|_{S_1^n\big[\ell_p^N(S_p^d)\big]}-\|\xi_p\|^p_{\ell_p^N(S_p^{dn})}}{p-1}\leq S\Big(\sum_{i=1}^N\beta_i^p (tr_d\otimes id_n)(\tilde{B}_i) \Big)+\epsilon.
\end{align}and 
\begin{align}\label{claim second term}
\frac{\|\xi_p\|^p_{\ell_p^N(S_p^{dn})}-1}{p-1}+ \frac{1-\|\rho_p\|_{\ell_p^N(S_p^d)\big[\ell_1^N(S_1^d)\big]}}{p-1}\leq  -S(\tilde{B_i})+S\big((id_d\otimes tr_d)(B_i)\big)+\epsilon.
\end{align} With these estimates at hand, we can easily conclude our proof, since we will have, for every $p-1<\delta$, 
\begin{align*}
\frac{\pi_q^o(\Ne^*)-1}{p-1}&\leq (1+\epsilon)\Big\{S\Big(\sum_{i=1}^N\beta_i^p \Ne_i\big((tr_d\otimes id_d)(B_i)\big) \Big)\\&+ \sum_{i=1}^N\beta_i^p\Big[S\Big( (id_d\otimes tr_d)(B_i)\Big)-S\Big((id_d \otimes \Ne_i)(B_i)\Big)\Big]+2\epsilon\Big\}.
\end{align*}If we denote $\lambda_i=\beta_i^p$ and $\eta_i=B_i$ for every $i=1,\cdots ,N$ we see that the previous expression is (up to the $\epsilon$'s) one of those appearing in the definition of $C_E((\Ne_i)_i)$. Using that $d$ and $n$ are fixed numbers and that $S(\rho)\leq \ln m$ for any state $\rho\in S_1^m$, since the previous estimate holds for an arbitrary $\epsilon> 0$, the result follows.

It remains to prove claims (\ref{claim first term}) and (\ref{claim second term}). For the first one we define
\begin{align*}
\Delta_p=\sum_{i=1}^N\beta_i^p (tr_d\otimes id_n)(\tilde{B}_i)\text{       } \text{     and      }\text{       } \Lambda_p=\sum_{i=1}^N\beta_i^p(tr_d\otimes id_n)(\tilde{B}_i^p),
\end{align*}which are positive elements in $M_n$ such that $tr(\Delta_p)=1$ and $tr(\Lambda_p)=\|\xi_p\|^p_{\ell_p^N(S_p^{dn})}$. Moreover, we know that $\|\Lambda_p\|\leq tr(\Lambda_p)\leq 1$. Therefore, we can apply functional calculus and Remark \ref{aprox log} to conclude that
\begin{align*}
\Lambda_p^\frac{1}{p}-\Lambda_p\leq (p-1)\Lambda_p^\frac{1}{p} (-\ln\Lambda_p^\frac{1}{p}).
\end{align*}
On the other hand, according to Lemma \ref{GenralLieb} and taking into account the $\text{flip}$ map in the definition of $\xi_p$ we have $$\|\xi_p\|_{S_1^n\big[\ell_p^N(S_p^d)\big]}\leq tr_n\Big[\Big(\sum_{i=1}^N\beta_i^p(tr_d\otimes id_n)(\tilde{B}_i^p)\Big)^\frac{1}{p}\Big].$$Hence,
\begin{align*}
\frac{\|\xi_p\|_{S_1^n\big[\ell_p^N(S_p^d)\big]}-\|\xi_p\|^p_{\ell_p^N(S_p^{dn})}}{p-1}\leq \frac{tr_n(\Lambda_p^\frac{1}{p}-\Lambda_p)}{p-1}\leq
S(\Lambda_p^\frac{1}{p}),
\end{align*}where we denote $S(\Lambda_p^\frac{1}{p})=tr\big(\Lambda_p^\frac{1}{p} (-\ln\Lambda_p^\frac{1}{p})\big)$. Therefore, it suffices to show that
\begin{align}\label{claim first term technical}
\big|S(\Delta_p)-S(\Lambda_p^\frac{1}{p})\big|<\epsilon.
\end{align}According to Theorem \ref{cont-Entropy}, there exists a $\gamma=\gamma(\epsilon,n)> 0$ so that $\|\Delta_p-\Lambda_p^\frac{1}{p}\|_1\leq\gamma$ and $\big| \|\Lambda_p^\frac{1}{p}\|_1-1\big|\leq \gamma$ implies (\ref{claim first term technical}). In fact, since $\|\Delta_p\|_1=1$, the second of these condition is implied by the first one. On the other hand, we can write
\begin{align*}
\|\Delta_p-\Lambda_p^\frac{1}{p}\|_1\leq \|\Delta_p-\Lambda_p\|_1+\|\Lambda_p-\Lambda_p^\frac{1}{p}\|_1.
\end{align*}Now, using conditions (\ref{normalization condition I}) and (\ref{normalization condition II}) for the first term and Remark \ref{aprox log} for the second one, it is very easy to see that we can find a $\delta=\delta(\gamma, n, d)=\delta(\epsilon,n,  d)$ so that $p-1<\delta$ implies that the previous quantity is smaller than $\gamma$. 
This proves (\ref{claim first term technical}), so (\ref{claim first term}).
 
In order to show (\ref{claim second term}), we first note that
\begin{align*}
\frac{\|\xi_p\|^p_{\ell_p^N(S_p^{dn})}-1}{p-1}+ \frac{1-\|\rho_p\|_{\ell_p^N(S_p^d)\big[\ell_1^N(S_1^d)\big]}}{p-1}\leq \sum_{i=1}^N\beta_i^p\left(\frac{\|\tilde{B_i}\|_{S_p^{dn}}^p-1}{p-1}+\frac{1-\|(id_d\otimes tr_d)(B_i)\|_{S_p^d}^p}{p-1}\right),
\end{align*}where we have used that $\sum_{i=1}^N\beta_i^p\|(id_d\otimes tr_d)(B_i)\|_{S_p^d}^p\leq\Big(\sum_{i=1}^N\beta_i^p\|(id_d\otimes tr_d)(B_i)\|_{S_p^d}^p\Big)^\frac{1}{p}$. Then, the basic idea to obtain our estimate (\ref{claim second term}) is nothing else than differentiating the new expression. Indeed, according to Theorem \ref{derivate-Entropy}, if we differentiate such an expression we should obtain $-S(\tilde{B_i})+S\big((id_d\otimes tr_d)(B_i)\big)$ for every $i$. The problem here is that we cannot consider the $\lim_p$ since $N$ also depends on $p$ and we must have finite dimensional states. Then, we just `differentiate by hand'' by using Remark \ref{aprox log}. Let $(\alpha_j^i(p))_{j=1}^{dn}$ be the eigenvalues of the state\footnote{We remark here the dependence on $p$ to see that there is no problem with that.} $\tilde{B_i}(p)$ for every $i=1,\cdots ,N$. Then,
\begin{align}\label{second term}
\frac{\|\tilde{B_i}(p)\|_{S_p^{dn}}^p-1}{p-1}&=\sum_{j=1}^{dn}\frac{(\alpha_j^i(p))^p-\alpha_j^i(p)}{p-1}\leq \sum_{j=1}^{dn}\alpha_j^i(p)^p\ln \alpha_j^i(p)\\& \nonumber=\sum_{j=1}^{dn}\big(\alpha_j^i(p)^p-\alpha_j^i(p)\big)\ln \alpha_j^i(p)-S(\tilde{B_i})\\&\nonumber
\leq\sum_{j=1}^{dn}\alpha_j^i(p)^p\big(\ln \alpha_j^i(p)\big)^2(p-1)-S(\tilde{B_i}).
\end{align}Hence, using that the $\alpha_j^i(p)\in [0,1]$ for every $i,j,p$ and the fact that the function $f(x)=x^p\ln(x)^2$ is unifomaly upper bounded in $[0,1]$, we can find a $\delta=\delta(\epsilon, n, d)> 0$ so that $p-1< \delta$ implies $\frac{\|\tilde{B_i}(p)\|_{S_p^{dn}}^p-1}{p-1}\leq -S(\tilde{B_i})+\epsilon$. On the other hand, one can also check that 
\begin{align*}
\frac{1-\|(id_d\otimes tr_d)(B_i)\|_{S_p^d}^p}{p-1}\leq S\big( (id_d\otimes tr_d)(B_i)\big).
\end{align*}Thus, we obtain (\ref{claim second term}).
\end{proof}
\begin{remark}\label{remark pure states}
It is interesting to note that $B_i(p)$ is a pure state associated to the element $b_i(p)\in \ell_2^{d^2}$ for every $i=1,\cdots, N(\epsilon, p)$. Indeed, if $B=\sum_{i,j=1}^nb_{i,j}|i\rangle\langle j|\in M_n$, the map $M_B:M_n\rightarrow M_n$ can be seen as the map associated to the tensor $|b\rangle\langle b|\in M_n\otimes M_n$, where $|b\rangle=\sum_{i,j=1}^nb_{i,j}|ij\rangle\in \ell_2^{n^2}$. Here, the trace duality between $S_1^n$ and $M_n$ is described by $\langle A,B\rangle=tr(AB^t)$, where $B^t$ is the transpose operator.
\end{remark}
To prove the converse inequality, $\tilde{C}(\Ne)\geq C_E((\Ne_i)_i)$, we will need the following two lemmas.
\begin{lemma}\label{aequalone}
Let $\A$ and $\B$ be finite dimensional $C^*$-algebras and let $\Ne :L_1(\A)\rightarrow L_1(\B)$ be a quantum channel such that $\Ne(\uno_\A)$ has full support (that is, it is an invertible element of $\B$). Let $T:\B\rightarrow \A$ be a completely positive contraction such that there exists a positive element $a$ in the unit ball of $\B$ verifying $$\Ne^*(x)=T(axa)$$ for every $x\in \B$. Then $a=\uno_\B$ and $\Ne^*=T$.
\end{lemma}
\begin{proof}
Using that $\Ne$ is trace preserving we have
\begin{align*}
tr_\A(\uno_\A)&=tr_{\B}\big(\Ne (\uno_\A)\big)=tr_{\B}\big(aT^*(\uno_\A)a\big)=tr_{\B}\big(T^*(\uno_\A)^\frac{1}{2}T^*(\uno_\A)^\frac{1}{2}a^2\big)
\\&\leq tr_{\B}\big(T^*(\uno_\A))^\frac{1}{2}tr_{\B}\big(T^*(\uno_\A)a^4)^\frac{1}{2}\leq tr_{\B}\big(T^*(\uno_\A)\big)=tr_{\B}\big(\uno_\B T^*(\uno_\A)\big)\\&=tr_\A(\uno_\A T(\uno_\B))\leq tr_\A(\uno_\A),
\end{align*}where the first inequality follows from Cauchy-Schwartz inequality. Moreover, we have used that, since the maps $M_a$ is self adjoint, $\Ne^*=T\circ M_a$ implies $\Ne=M_a\circ T^*.$ Now, using that full support of $\Ne(\uno_\A)$ implies full support of $T^*(\uno_\A)$, equality in Cauchy-Schwartz inequality implies that $a^2=\uno_{\B}$. Since $a$ is positive we conclude $a=\uno_{\B}$.
\end{proof}
\begin{remark}\label{aequalone- Remark}
Given a quantum channel $\Ne:L_1(\A)\rightarrow L_1(\B)$ between finite dimensional von Neumann algebras, we can always assume that $\Ne(\uno_\A)$ has full support. Otherwise, we consider the finite dimensional von Neumann algebra $\tilde{\B}=p\B p$, where $p$ is the support projection of $\Ne(\uno_\A)$, and consider the new quantum channel $\Ne:L_1(\A)\rightarrow L_1(\tilde{\B})$.
\end{remark}
\begin{lemma}\label{limitetrace}
Let $(a(p))_p$ be a net of positive and invertible operators in $M_n$ verifying the following properties:
\begin{enumerate}
\item[1)] $\sup_p\|a(p)^{-1}\|_{M_n}\leq M$ for a certain positive constant $M$, and
\item[2)] $\lim_{p\rightarrow 1}\ln \|a(p)\|_{S_{q}^n}^{q}=0$, where $\frac{1}{p}+\frac{1}{q}=1$.
\end{enumerate}
Then,
\begin{align*}
\liminf_{p\rightarrow 1^+}\frac{tr\big(a(p)^{-1}\rho\big)-1}{p-1}\geq S(\rho)
\end{align*}for every density operator $\rho$.
\end{lemma}
\begin{proof}
According to Theorem \ref{derivate-Entropy} we have,
\begin{equation*}
\lim_{p\rightarrow 1}\frac{\|\rho\|_p-1}{1-p}=S(\rho).
\end{equation*}This implies
\begin{equation}\label{derivate adapted}
\lim_{p\rightarrow 1}\frac{\|\rho\|_{\frac{p}{2p-1}}-1}{p-1}=S(\rho).
\end{equation}On the other hand, for $p> 1$ we can write
\begin{align}\label{simpleII}
\frac{\|\rho\|_{\frac{p}{2p-1}}-1}{p-1}=\frac{\|a(p)^\frac{1}{2}a(p)^{-\frac{1}{2}}\rho a(p)^{-\frac{1}{2}}a(p)^\frac{1}{2}\|_{\frac{p}{2p-1}}-1}{p-1}\leq \frac{\|a(p)\|_{q}\|a(p)^{-\frac{1}{2}}\rho a(p)^{-\frac{1}{2}}\|_1-1}{p-1},
\end{align}where we have used the non-commutative generalized Holder's inequality (see \cite[Section 1]{PiXu}) with $\frac{1}{\frac{p}{(2p-1)}}=1+ \frac{1}{\frac{p}{p-1}}=1+ \frac{1}{q}$.

Since we have
\begin{align*}
\frac{\|a(p)\|_{q}\|a(p)^{-\frac{1}{2}}\rho a(p)^{-\frac{1}{2}}\|_1-1}{p-1}=\big(\frac{\|a(p)\|_{q}-1}{p-1}\big)tr(a(p)^{-1}\rho)+\frac{tr(a(p)^{-1}\rho)-1}{p-1}
\end{align*}for every $p$, we will conclude our proof from Equations (\ref{derivate adapted}) and (\ref{simpleII}) by showing $$\lim_{p\rightarrow 1}\Big(\frac{\|a(p)\|_{q}-1}{p-1}\Big)tr\big(a(p)^{-1}\rho\big)=0.$$To this end, note that
\begin{align*}
\lim_{p\rightarrow 1}\Big|\Big(\frac{\|a(p)\|_{q}-1}{p-1}\Big)tr\big(a(p)^{-1}\rho\big)\Big|\leq M\lim_{p\rightarrow 1}\Big|\frac{\|a(p)\|_{q}-1}{p-1}\Big|=M\lim_{p\rightarrow 1}\frac{q}{p}\big(\|a(p)\|_{q}-1\big)\\=M\lim_{p\rightarrow 1}\frac{q}{p}\big(e^{\frac{1}{q}\ln \|a(p)\|_{q}^{q}}-1\big)=M\lim_{p\rightarrow 1}\frac{1}{p}(\ln \|a(p)\|_{q}^{q})=0,
\end{align*}where we have used that $e^{\frac{1}{q}\ln \|a(p)\|_{q}^{q}}\simeq 1+ \frac{1}{q}\ln \|a(p)\|_{q}^{q}$ when $p$ is close to $1$.
\end{proof}
We are now ready to prove the second inequality in Theorem \ref{mainII}.
\begin{proof}[Proof of inequality $\geq$ in Theorem \ref{mainII}]\label{second-inequality}

Let $\Upsilon=\big\{(\lambda_i)_{i=1}^N, ,(\eta_i)_{i=1}^N\big\}$ be an ensemble optimizing $C_E((\Ne_i)_i)$. We must show that
\begin{align*}
\lim_{q\rightarrow \infty}q\ln \big(\Pi_q^o(\Ne^*)\big)\geq \sup\Big\{S\Big(\sum_{i=1}^N \lambda_i(tr_d\otimes id_n)\big((id_d\otimes \Ne_i)(\eta_i)\big)\Big)\\ +\sum_{i=1}^N \lambda_i\Big[S\Big((id_d\otimes tr_n)\big((id_d\otimes \Ne_i)(\eta_i)\big)\Big)-S\Big((id_d\otimes \Ne_i)(\eta_i)\Big) \Big] \Big\}.
\end{align*} Clearly, it suffices to prove the previous inequality if we consider the new channel defined by restricting $\Ne$ to $\ell_1^N(S_1^d)$. We will use the same notation $\Ne:\ell_1^N(S_1^d)\rightarrow S_1^n$ for the restricted channel. Moreover, as we explained in Remark \ref{aequalone- Remark} we can assume that $\Ne(\uno_{\ell_1^N(S_1^d)})$ has full support.

According to Theorem \ref{factorization-positivity} for every $1<q$ we can consider an optimal factorization $$\Ne^*=T_qM_{a(q)},$$ where $T_q:S_q^n\rightarrow \ell_\infty^N(M_d)$ is a completely positive map, $M_{a(q)}:M_n\rightarrow S_q^n$ is the associated operator to a certain positive element $a(q)\in M_n$ and such that $\|a(q)\|_{2q}^{2}\|T_q\|_{cb}=\Pi_q^o(\Ne^*)$. Here, $\frac{1}{p}+\frac{1}{q}=1$. Furthermore, by rescaling we may assume
\begin{align*}
\|a(q)\|^2_{2q}=\pi_{q}^o(\Ne^*)^\frac{1}{q}, \text{     }\text{     }\text{     } \|T_q:S_q^n\rightarrow \ell_\infty^N(M_d)\|_{cb}=\pi_{q}^o(\Ne^*)^\frac{1}{p}.
\end{align*}Actually, the fact that $\Ne(\uno_{\ell_1^N(S_1^d)})=a(q)T^*_q(1_{\ell_1^N(S_1^d)})a(q)$ has full support guarantees that $a(q)$ is also invertible for every $q$. By continuity, we deduce that
\begin{equation}\label{convergence}
\lim_{q\rightarrow \infty}\|a(q)\|_{2q}^{2q}=\lim_{q\rightarrow \infty}\pi_{q}^o(\Ne^*)=\|\Ne^*\|_{cb}=1.
 \end{equation}
On the other hand,
\begin{align*}
\|T_q:M_n\rightarrow \ell_\infty^N(M_d)\|_{cb}\leq \|id_n:M_n\rightarrow S_{q}^n\|_{cb}\|T_q: S_{q}^n\rightarrow \ell_\infty^N(M_d)\|_{cb}\leq n^\frac{1}{q}\pi_{q}^o(\Ne^*)^\frac{1}{p}.
\end{align*}By a compactness argument we can assume that $$\lim_{q\rightarrow \infty}T_q=T:M_n\rightarrow \ell_\infty^N(M_d),$$where $T$ is a completely positive and completely contractive map. In the same way, we see that
\begin{align*}
\|a(q)\|\leq \|a(q)\|_{2q}=\pi_{q}^o(\Ne^*)^{\frac{1}{2q}};
\end{align*}so we can assume that $$\lim_{q\rightarrow \infty}a(q)=a,$$ where $a$ is a positive operator in $M_n$ verifying $0\leq a\leq \uno_{M_n}$.

It follows by construction that $\Ne^*(x)=T(axa)$ for every $x\in M_n$. Moreover, we can apply Lemma \ref{aequalone} to conclude that $a=\uno_{M_n}$ and $T=\Ne^*$. This implies, in particular, that $\lim_{q\rightarrow \infty}a(q)^{-1}=1$. Considering a subnet we can assume that $\sup_q{\|a(q)^{-1}\|}\leq M$ for a positive constant $M$. Now, Lemma \ref{quotient} and Equation (\ref{convergence}) allow us to write
\begin{align*}
\lim_{q\rightarrow \infty}q\ln \big(\Pi_q^o(\Ne^*)\big)=\lim_{q\rightarrow \infty}q \Big( \ln\big(\|a(q)\|^2_{2q}\big)+ \ln \big(\|T_q:S_q^n\rightarrow \ell_\infty^N(M_d)\|_{cb}\big)\Big )
\\=\lim_{q\rightarrow \infty}q \ln \|T_q\|_{cb}=\lim_{q\rightarrow \infty}q \ln \|T^*_q:\ell_1^N(S_1^d)\rightarrow S_p^n\|_{cb}=\lim_{p\rightarrow 1}\frac{\|T^*_q\|_{cb}-1}{p-1}.
\end{align*}In order to simplify notation we will denote $T_p=T^*_q:\ell_1^N(S_1^d)\rightarrow S_p^n$. Now, note that
\begin{align*}
\lim_{p\rightarrow 1}\frac{\|T_p\|_{cb}-1}{p-1}\geq \lim_{p\rightarrow 1}\frac{\big\|id_{\ell_\infty^N(M_d)}\otimes T_p:\ell_p^{N}(S_p^d)[\ell_1^N(S_1^d)]\rightarrow \ell_p^N(S_p^d)[S_p^n]\big\|-1}{p-1}\\
\geq \lim_{p\rightarrow 1}\frac{1}{p-1}\Big(\frac{\|(id\otimes T_p)(\rho)\|_{\ell_p^N(S_p^{dn})}}{\|\rho\|_{\ell_p^N(S_p^d)[\ell_1^N(S_1^d)]}}-1\Big),
\end{align*} where $\rho=\sum_{i=1}^N\lambda_ie_i\otimes e_i\otimes \eta_i\in \ell_1^{N^2}(S_1^{d^2})$ is the state defined by the ensemble $\Upsilon$ that we have considered at the beginning of the proof.

Let us denote $\tilde{\xi}_p=(id_{\ell_\infty^N(M_d)}\otimes T_p)(\rho)=\sum_{i=1}^N\lambda_ie_i\otimes (id_d\otimes T_p^i)(\eta_i)\in \ell_p^N(S_p^{dn})$, where $T_p=(T_p^i:S_1^d\rightarrow S_p^n)_{i=1}^N$. The previous expression can be written as
\begin{align*}
\lim_{p\rightarrow 1}\frac{1}{p-1}\Big(\frac{\|\tilde{\xi}_p\|_p}{\|\rho\|_{(p,1)}}-1\Big)
=\lim_{p\rightarrow 1}\frac{1}{\|\rho\|_{(p,1)}}\Big(\frac{\|\tilde{\xi}_p\|_p-1+1-\|(id\otimes tr)(\rho)\|_{\ell_p^N(S_p^d)}}{p-1}\Big),
\end{align*}where we have used (\ref{(p,1)}) to write $$\|\rho\|_{(p,1)}=\|(id\otimes tr)(\rho)\|_{\ell_p^N(S_p^d)}=\Big\|\sum_{i=1}^N\lambda_ie_i\otimes (id_d\otimes tr_d)(\eta_i)\Big\|_{\ell_p^N(S_p^d)}.$$
Summarizing the previous steps, we have that
\begin{align*}
\lim_{q\rightarrow \infty}q\ln \big(\Pi_q^o(\Ne^*)\big)\geq \lim_{p\rightarrow 1}\frac{1}{\|\rho\|_{(p,1)}}\Big(\frac{\|\tilde{\xi}_p\|_p-1+1-\|(id\otimes tr)(\rho)\|_{\ell_p^N(S_p^d)}}{p-1}\Big).
\end{align*}
Now, the fact that $\rho$ is a state guarantees
\begin{align}\label{prop2 factor term}
\lim_{p\rightarrow 1}\|\rho\|_{(p,1)}=1.
 \end{align}On the other hand, according to Theorem \ref{derivate-Entropy} and the definition of the von Neumann entropy we also have
\begin{align}\label{prop2 second term}
\lim_{p\rightarrow 1}\frac{1-\|(id\otimes tr)(\rho)\|_{\ell_p^N(S_p^d)}}{p-1}=S\Big(\sum_{i=1}^N\lambda_ie_i\otimes (id_d\otimes tr_d)(\eta_i)\Big)\\\nonumber=\sum_{i=1}^N\lambda_iS\Big((id_d\otimes tr_d)(\eta_i)\Big) + H\big((\lambda_i)_{i=1}^N\big).
\end{align}Here, $H\big((\lambda_i)_{i=1}^N\big):=-\sum_{i=1}^N\lambda_i\ln \lambda_i$ is the classical (Shannon) ($\ln-$) entropy of the probability distribution $(\lambda_i)_{i=1}^N$. Therefore, the proof of the theorem will follow from (\ref{prop2 factor term}), (\ref{prop2 second term}) and the estimate
\begin{align}\label{semifinal}
\frac{\|\tilde{\xi}_p\|_{\ell_p^N(S_p^{dn})}-1}{p-1}&\geq -\sum_{i=1}^N\lambda_iS\big((id_d\otimes \Ne_i)(\eta_i)\big)- H((\lambda_i)_{i=1}^N)\\\nonumber&+S\Big(\sum_{i=1}^N\lambda_i \Ne_i\big((tr_d\otimes id_d)(\eta_i)\big)\Big).
\end{align}In order to show this last estimate, we define the state $\xi_p=\frac{\tilde{\xi}_p}{\|\tilde{\xi}_p\|_{\ell_1^N(S_1^{dn})}}$ and then write
\begin{align*}\frac{\|\tilde{\xi}_p\|_{\ell_p^N(S_p^{dn})}-1}{p-1}=\|\tilde{\xi}_p\|_1\Big(\frac{\|\xi_p\|_p-1}{p-1}\Big)+ \frac{\|\tilde{\xi}_p\|_1-1}{p-1}.
\end{align*}
Now, according to our construction
\begin{align*}
\lim_{p\rightarrow 1}\tilde{\xi}_p=\lim_{p\rightarrow 1}(id_{\ell_\infty^N(M_d)}\otimes T_q^*)(\rho)=(id_{\ell_\infty^N(M_d)}\otimes T^*)(\rho)=(id_{\ell_\infty^N(M_d)}\otimes \Ne)(\rho),
\end{align*}and
\begin{align}\label{prop2 first term1}
\lim_{p\rightarrow 1}\|\tilde{\xi}_p\|_{\ell_1^N(S_1^{dn})}=1.
\end{align}On the other hand, Theorem \ref{derivate-Entropy} says that
\begin{align}\label{prop2 first term2}
\lim_{p\rightarrow 1}\frac{\|\xi_p\|_p-1}{p-1}=- S\big((id_{\ell_\infty^N(M_d)}\otimes \Ne)(\rho)\big)=-S\Big(\sum_{i=1}^N\lambda_i e_i\otimes (id_d\otimes \Ne_i)(\eta_i)\Big)\\\nonumber=-\sum_{i=1}^N\lambda_iS\big((id_d\otimes \Ne_i)(\eta_i)\big)- H((\lambda_i)_{i=1}^N).
\end{align}
Finally, (\ref{convergence}) allows us to apply Lemma \ref{limitetrace} to the net $(a(p)^2)_p$ to obtain
\begin{align}\label{prop2 first term3}
\lim_{p\rightarrow 1}\frac{\|\tilde{\xi}_p\|_1-1}{p-1}&=\lim_{p\rightarrow 1}\frac{tr_n\Big(T_p\big(\sum_{i=1}^N\lambda_ie_i\otimes (tr_d\otimes id_d)(\eta_i)\big)\Big)-1}{p-1}\\&\nonumber=\lim_{p\rightarrow 1}\frac{tr_n\Big(a(p)^{-2}\Ne\big(\sum_{i=1}^N\lambda_ie_i\otimes (tr_d\otimes id_d)(\eta_i)\big)\Big)-1}{p-1}%\\&=\lim_{p\rightarrow 1}\frac{tr_n\Big(a(p)^{-2}\big(\sum_{i=1}^N\lambda_i \Ne_i\big((tr_d\otimes 1_{M_n})(\eta_i)\big)\big)\Big)-1}{p-1}
\\&\nonumber\geq S\Big(\sum_{i=1}^N\lambda_i \Ne_i\big((tr_d\otimes id_d)(\eta_i)\big)\Big).
\end{align}
The estimate (\ref{semifinal}) follows now easily from (\ref{prop2 first term1})-(\ref{prop2 first term3}). This concludes the proof.
\end{proof}
\begin{remark}
Actually, we have shown that the states $\eta_i$'s and the probabilities $\lambda_i$'s in the expression
\begin{align*}
\sup\Big\{S\Big(\sum_{i=1}^N \lambda_i (\Ne\circ \phi_i)((tr_d\otimes id_d)(\eta_i))\Big)
+\sum_{i=1}^N \lambda_i\Big[S\Big((id_d\otimes tr_d)(\eta_i)\Big)\\-S\Big(\big(id_d\otimes (\Ne\circ \phi_i)\big)(\eta_i)\Big)\Big] \Big\}
\end{align*} in Theorem \ref{mainII} are given by Theorem \ref{factorization-positivity}. This means that the factorization theorem tells us the objects that we have to use in order to attain the capacity of the channel. In particular, according to Remark \ref{remark pure states} we have shown that considering pure states $\eta_i$ in the expression (\ref{d-restricted capacity}) is not a restriction, but it covers the general case.
\end{remark}
\begin{remark}\label{Remark on classical channels}[Classical channels]
As we pointed out in the introduction, it is well known that $C_{prod}^d(\Ne)$ coincides with $C_c(\Ne)$ for every $d$ and for every classical channel $\Ne:S_1^n\rightarrow S_1^n$ (this means that entanglement cannot increase the classical capacity of a classical channel). On the other hand, it is easy to see that $\pi_p^o(T:\ell_\infty^n\rightarrow \ell_\infty^n)=\pi_p(T:\ell_\infty^n\rightarrow \ell_\infty^n)$ for every $T:\ell_\infty^n\rightarrow \ell_\infty^n$. Indeed, one way of seeing this is by invoking the factorization theorem for absolutely $p$-summing maps (resp. completely $p$-summing maps) and using that $\big\|T:X\rightarrow \ell_\infty\big\|_{cb}=\big\|T:X\rightarrow \ell_\infty\big\|$   for every linear map $T$ and every operator space $X$. Thus, in this case we recover (\ref{mainI-classical channels}).
\end{remark}
\begin{remark}\label{The cases $d=1$ and $d=n$}[The cases $d=1$ and $d=n$]
Given a quantum channel $\Ne:S_1^n\rightarrow S_1^n$, $C_{prod}^1(\Ne)$ and $C_{prod}^n(\Ne)$ coincide, respectively, with the Holevo capacity and the unlimited entanglement-assisted classical capacity of $\Ne$.

To see the first one, we just write the expression in Theorem \ref{mainI} for $d=1$ and we obtain
\begin{align*}
C_{prod}^1(\Ne)=\sup\Big\{S\big(\sum_{i=1}^N \lambda_i \Ne (\xi_i)\big)
-\sum_{i=1}^N \lambda_i S\big(\Ne (\xi_i)\big) \Big\},
\end{align*}where the supremum runs over all $N\in \N$, all probability distributions $(\lambda_i)_{i=1}^N$ and all families $(\xi_i)_{i=1}^N$, with $\xi_i$ state in $S_1^n$ for every $i=1,\cdots, N$. This is exactly the expression of the Holevo capacity of $\Ne$ (see Theorem \ref{Theorem Holevo} in Section \ref{Phy Int}).

The key point to study the case $d=n$ is to realize that we do not need to consider the embedding $j_n:M_n\hookrightarrow \ell_\infty(\mathcal P, M_n)$. First of all, let us recall that $\pi_{q,n}(\Ne^*)=\pi_q^o(\Ne^*)$ for every quantum channel $\Ne^*:M_n\rightarrow M_n$, which follows from the definition of the norms (see Section \ref{Pisier's theorem Section}). Then, using that $j_n$ is a complete isometry on positive elements and the good behavior of $\pi_q^o$ with respect to positivity shown in Theorem \ref{factorization-positivity}, one easily has
\begin{align*}
\pi_q^o(j_n\circ \Ne^*)= \pi_q^o(\Ne^*)=\pi_{q,n}(\Ne^*)
\end{align*}for every quantum channel $\Ne:S_1^n\rightarrow S_1^n$. Therefore, in this case ($d=n$) Theorem \ref{mainI} is obtained from Theorem \ref{mainII} applied to the single channel $\Ne$ instead of on a family of infinitely many channels $(\Ne_i)_i$. Then, we have
\begin{align*}
C_{prod}^n(\Ne)=\lim_{q\rightarrow \infty}q\ln \big(\Pi_q^o(\Ne^*)\big)=\sup\Big\{S\Big(\Ne\big((tr_n\otimes id_n)(\eta)\big)\Big)\\ +S\Big((id_n\otimes tr_n)(\eta)\Big)-S\Big((id_n\otimes \Ne)(\eta)\Big)  \Big\},
\end{align*} where the supremum runs over all pure states $\eta\in S_1^n\otimes S_1^n$. This is exactly the expression of $C_E(\Ne)$ (see Theorem \ref{Theorem BSST} in Section \ref{Phy Int}).
\end{remark}
\section{Covariant channels and non additivity of $C_{prod}^d$}\label{Section: non-additivity}
In this section we will discuss a particularly nice kind of quantum channels called covariant channels. We will see that the factorization theorem has a very simple form for these channels. As a direct consequence of this fact, we will show that there is an easy relation between the (unlimited) entanglement-assisted classical capacity $C_E$ of a covariant channel and the \emph{cb-min entropy} of a quantum channel introduced in \cite{DJKR}. In the second part of this section, we will use our results on covariant channels to prove Theorem \ref{counterexample additivity}. As we will explain in Section \ref{Phy Int} a direct consequence of this theorem is that the product state capacity of the $d$-restricted capacity, $C_{prod}^d$, does not coincide, in general, with its regularization version for $1<d< n$.
\begin{definition}\label{Defi- covariant}
Let $G$ be a topological compact group and let us consider representations $\pi, \sigma:G\rightarrow \mathbb U(n)$, where $\mathbb U(n)$ denotes the unitary group in dimension $n$. We say that a quantum channel $\Ne:S_1^n\rightarrow S_1^n$  is \emph{covariant} (with respect to ($G,\pi,\sigma$)) if
\begin{enumerate}
\item[1.] $\int_G\pi(g)x\pi(g^*) dg=\frac{1}{n}tr_n(x)\uno_n$ for every $x\in M_n$.
\item[2.] $\Ne(\pi(g)x\pi(g^*))=\sigma(g)\Ne(x) \sigma(g^*)$ for every $x\in M_n$ and $g\in G$.
\end{enumerate}
Here, the integral is with respect to the Haar measure of the group.
\end{definition}
The following result is an easy consequence of a Pisier's version of the Wigner-Yanase-Dyson inequalities (see \cite[Lemma 1.14]{Pisierbook2}).
\begin{lemma}\label{Covariant}
Let $T:M_n\rightarrow M_n$ be a linear map which is covariant with respect to ($G,\pi,\sigma$). Then, for any $1\leq d\leq n$ we have
\begin{align*}
\pi_{p,d}(T)=n^\frac{1}{p}\|T:S_p^n\rightarrow M_n\|_{d}.
 \end{align*}In the case $d=1$ we obtain $\pi_{p,cb}(T)=n^\frac{1}{p}\|T:S_p^n\rightarrow M_n\|$ while for $d=n$ we have $\pi_p^o(T)=n^\frac{1}{p}\|T:S_p^n\rightarrow M_n\|_{cb}$.
\end{lemma}
\begin{proof}
To prove inequality $\leq$ just note that
\begin{align*}
\pi_{p,d}(T:M_n\rightarrow M_n)\leq \pi_{p,d}(id_n:M_n\rightarrow S_p^n)\|T:S_p^n\rightarrow M_n\|_{d}\leq n^\frac{1}{p}\|T:S_p^n\rightarrow M_n\|_{d}.
\end{align*}Here, the first inequality follows from the very definition of the $\pi_{p,d}$ norm while the second inequality can be obtained by considering the trivial factorization in the factorization theorem of $\pi_{p,d}(id_n:M_n\rightarrow S_p^n)$.
For the converse inequality let us fix $1\leq d\leq n$ and assume that $\pi_{p,d}(T)=1$. We will conclude our proof if we show that $\|T:S_p^n\rightarrow M_n\|_d\leq n^{-\frac{1}{p}}$. Now, according to the factorization theorem there exist positive elements $a,b\in M_n$ verifying $\|a\|_{2p}=\|b\|_{2p}=1$ such that
\begin{align*}
\|(id_d\otimes T)(x)\|_{S_p^d[M_n]}\leq \|(\uno\otimes a)x(\uno\otimes b)\|_{S_p(\ell_2^d\otimes\ell_2^n)}
\end{align*} for every $x\in M_{dn}$.

Now, for every $x\in M_{dn}$ we have
\begin{align*}
\|(id_d\otimes T)(x)\|_{S_p^d[M_n]}&= \big\|(\uno\otimes \sigma(g))\big(id\otimes T)(x)\big)(\uno\otimes \sigma(g^*))\big\|_{S_p^d[M_n]}\\&
=\big\|(id\otimes T)\big((\uno\otimes \pi(g))x(\uno\otimes \pi(g^*))\big)\big\|_{S_p^d[M_n]}\\&
\leq \big\|(\uno\otimes a\pi(g))x(\uno\otimes \pi(g^*)b)\big\|_{S_p(\ell_2^d\otimes\ell_2^n)}
\end{align*}for every $g\in G$. Therefore, according to \cite[Lemma 1.14]{Pisierbook2} we obtain
\begin{align*}
&\|(id_d\otimes T)(x)\|_{S_p^d[M_n]}\leq \Big(\int_G\big\|(\uno\otimes a\pi(g))x(\uno\otimes \pi(g^*)b)\big\|_{S_p(\ell_2^d\otimes\ell_2^n)}^pdg\Big)^\frac{1}{p}
\\&\leq \Big\|\Big(\uno\otimes \big(\int_G(\pi(g^*)a\pi(g))^{2p}dg\big)^\frac{1}{2p}\Big)x \Big(\uno\otimes \big(\int_G(\pi(g^*)b\pi(g))^{2p}dg\big)^\frac{1}{2p}\Big)\Big\|_{S_p(\ell_2^d\otimes\ell_2^n)}\\&
=\Big\|\Big(\uno\otimes \big(\int_G \pi(g^*)a^{2p}\pi(g)dg\big)^\frac{1}{2p}\Big)x \Big(\uno\otimes \big(\int_G \pi(g^*)b^{2p}\pi(g)dg\big)^\frac{1}{2p}\Big)\Big\|_{S_p(\ell_2^d\otimes\ell_2^n)}\\&
=n^{-\frac{1}{p}}\|x\|_{S_p^d[S_p^n]}.
\end{align*}
This concludes the proof.
\end{proof}
The previous lemma says that if we are dealing with a covariant channel $\Ne$ we can always take $a=n^{-\frac{1}{2p}}\uno_n$ in the factorization given by Theorem \ref{factorization-positivity}. Thus, in order to compute $C_{prod}^d(\Ne)$ for these kinds of channels we will have to differentiate the norm $\|\Ne:S_1^n\rightarrow S_p^n\|_d$ instead of the $\pi_{q,d}(\Ne^*)$-norm. Indeed, we have
\begin{corollary}
For any covariant quantum channel $\Ne:S_1^n\rightarrow S_1^n$ we have
\begin{align*}
C_{prod}^d(\Ne)=\ln n+ \frac{d}{dp}[\|\Ne:S_1^n\rightarrow S_p^n\|_{d}]|_{p=1}
\end{align*}for every $1\leq d\leq n$.
\end{corollary}
\begin{proof}
First of all, note that $\Ne^*$ also verifies condition 2. in Definition \ref{Defi- covariant}. Therefore, applying Lemma \ref{Covariant} we obtain
\begin{align*}
C_{prod}^d(\Ne)=\lim_{q\rightarrow \infty}q \ln \pi_{p,d}(\Ne^*)&=\ln n+ \lim_{q\rightarrow \infty}q \ln \|\Ne^*:S_q^n\rightarrow M_n\|_{d}\\&=\ln n+ \lim_{q\rightarrow \infty}q \ln \|\Ne:S_1^n\rightarrow S_p^n\|_{d}\\&=\ln n+\frac{d}{dp}[\|\Ne:S_1^n\rightarrow S_p^n\|_{d}|]_{p=1}.
\end{align*}
\end{proof}
In particular, in this case we have an easy relation between the (unlimited) entanglement-assisted classical capacity of a quantum channel, $C_E(\Ne)$, and the cb-min entropy of $\Ne$ introduced in \cite{DJKR}: $$C_{CB,min}(\Ne):=-\frac{d}{dp}[\|\Ne:S_1^n\rightarrow S_p^n\|_{cb}]|_{p=1}.$$ We obtain that for every covariant quantum channel the equality
\begin{align*}
C_E(\Ne)=\ln n- C_{CB,min}(\Ne)
\end{align*}holds.
As we promised, we finish this section by proving Theorem \ref{counterexample additivity}.
In the proof, we will use the following result, which can be found in \cite{JuPa}.
\begin{theorem}\label{theorem exact capacity dep channel}
Let us consider the quantum depolarizing channel with parameter $\lambda\in [0,1]$, $\mathcal D_\lambda:S_1^n\rightarrow S_1^n$, defined by $$\mathcal D_\lambda(\rho)=\lambda\rho+ (1-\lambda)\frac{1}{n}tr(\rho)\uno_n \text{      }\text{    for every    }\text{      }  \rho\in S_1^n,$$and let $d$ be any natural number such that $1\leq d\leq n$. Then, 
\begin{align*}
\lambda\ln  (nd)-\ln 2\leq C_{prod}^d(\mathcal D_\lambda)\leq \lambda\ln  (nd).
\end{align*}
\end{theorem}
In fact, in \cite{JuPa} the exact value of $C_{prod}^d(\mathcal D_\lambda)$ is computed and Theorem \ref{theorem exact capacity dep channel} is a much weaker result. However, it will be enough for our purposes.
\begin{proof}[Proof of Theorem \ref{counterexample additivity}]

\

Let $\Ne_1:S_1^n\rightarrow S_1^n$ be the quantum channel defined by $$\Ne_1(\rho)=\sum_{i=1}^n tr(\rho e_{i,i})e_{i,i}$$ for every $\rho\in S_1^n$, where here $e_{i,i}$ denotes the ($n\times n$)-matrix with all entries equal to zero up to the entry $(i,i)$ which equals one. Note that this channel can be seen as
\begin{align*}
\Ne_1=i\circ id_{\ell_1^n}\circ P:S_1^n\rightarrow \ell_1^n\rightarrow \ell_1^n\hookrightarrow S_1^n,
\end{align*}where $P$ is the projection of $S_1^n$ onto the diagonal matrices and $i$ is the inclusion of diagonal matrices in $S_1^n$. That is, $\Ne_1$ is the classical identity (so the identity on $\ell_1^n$) regarded as a quantum channel. Hence, we have, $$C_{prod}^d(\Ne_1)=\ln n$$ for every $1\leq d\leq n$. Indeed, using Theorem \ref{mainI} this is immediate from the fact that
\begin{align*}
C_{prod}^d(\Ne_1)=\lim_{q\rightarrow \infty}q\ln\pi_{q,d}(\Ne_1^*)=\lim_{q\rightarrow \infty}q\ln\pi_q(\Ne_1^*)=\lim_{q\rightarrow \infty}q\ln n^\frac{1}{q}=\ln n.
\end{align*}Here, we have used that $\pi_{q,d}(\Ne_1^*)=\pi_q(\Ne_1^*)=n^\frac{1}{q}$. Indeed, the fact that $P^*$ and $i^*$ are complete contractions with $P^*\circ i^*=\uno_n$ and $i^*\circ P^*=\uno_{\ell_\infty^n}$ joint with the comments in Remark \ref{Remark on classical channels} guarantee that $$\pi_{q,d}(\Ne_1^*)=\pi_{q,d}(P^*\circ id_{\ell_\infty^n}\circ i^*)=\pi_{q,d}(id_{\ell_\infty^n})=\pi_q(id_{\ell_\infty^n})=n^\frac{1}{q}.$$
On the other hand, we will consider the depolarizing channel $\Ne_2^\lambda:S_1^n\rightarrow S_1^n$ defined by
\begin{align*}
\Ne_2^\lambda(\rho)=D_\lambda(\rho)= \lambda \rho+ (1-\lambda)tr(\cdot)\frac{\uno_n}{n}
\end{align*}for every $\rho\in S_1^n$. According to Theorem \ref{theorem exact capacity dep channel}, if we fix $\lambda=\frac{2}{3}\in (0,1)$ and $1< d=\sqrt{n} < n$ we have
\begin{align*}
\ln n -\ln 2\leq C_{prod}^{\sqrt{n}}(\mathcal \Ne_2^{\frac{2}{3}})\leq \ln n.
\end{align*}From this point on, we will assume that $\lambda$ is fixed and we will remove its dependence on $\Ne_2$.
Note that the following estimate holds for every pair of channels $\Ne_1$ and $\Ne_2$ and every $d$ with $1\leq d^2\leq n$.
\begin{align}\label{non-addit formula}
C_{prod}^{d^2}(\Ne_1\otimes \Ne_2)\geq C_{prod}^1(\Ne_1)+ C_{prod}^{d^2}(\Ne_2).
\end{align}This inequality follows from the fact that if we are using entanglement dimension $d^2$ the capacity is greater than or equal to the capacity given by the specific protocol in which Alice and Bob use all the entanglement in the second channel and they use independently the first channels without using any entanglement. Formally, we always have
\begin{align*}
\pi_{q,d^2}(\Ne_1^*\otimes \Ne_2^*)=\big\|id\otimes (\Ne_1^*\otimes \Ne_2^*):\ell_q(S_q^{d^2})\otimes_{min} (M_n\otimes_{min}M_n)\rightarrow \ell_q(S_q^{d^2})[M_n\otimes_{min}M_n]\big\|\\ \geq  \big\|id\otimes \Ne_1^*:\ell_q\otimes_{min} M_n\rightarrow \ell_q[M_n]\big\|\big\|id\otimes \Ne_2^*:\ell_q(S_q^{d^2})\otimes_{min} M_n\rightarrow \ell_q(S_q^{d^2})[M_n]\big\|,\end{align*}which equals $\pi_{q,1}(\Ne_1^*)\pi_{q,d^2}(\Ne_2^*)$. Indeed, this can be obtained by restricting to product elements in $$\ell_q(S_q^{d^2})\otimes_{min} (M_n\otimes_{min}M_n)=(\ell_q\otimes_q (\ell_q(S_q^{d^2}))\otimes_{min} (M_n\otimes_{min}M_n).$$Then, (\ref{non-addit formula}) follows directly from the relation $C_{prod}^d(\Ne)=\lim_{q\rightarrow \infty}q\ln \pi_{q,d}(\Ne^*)$ proved in or main Theorem \ref{mainI}. Hence, by considering our particular choice $\lambda=\frac{2}{3}$ and $d=\sqrt{n}$, we have 
\begin{align}\label{Equation 1 nonadditivity I}
C_{prod}^{n}(\Ne_1\otimes \Ne_2)\geq C_{prod}^1(\Ne_1)+ C_{prod}^{n}(\Ne_2)\geq \ln n +\frac{4}{3}\ln n-\ln 2=2\ln n +\frac{1}{3}\ln n-\ln 2.
\end{align}Let us recall that $C^{\sqrt{n}}_{prod}(\Ne_1)=\ln n$ and $\ln n- \ln2\leq C^{\sqrt{n}}_{prod}(\Ne_2)\leq\ln n$, so the previous expression gives a counterexample for the additivity of $C_{prod}^{d}$ by using different channels $\Ne_1$ and $\Ne_2$. Let us show how to find our channel $\Ne$. It is very easy to see that $\Ne_2$ is a covariant channel (with respect to ($G=\mathbb U (n),\pi=id,\sigma=id$)). According to Lemma \ref{Covariant} we then have
\begin{align*}
C_{prod}^{\sqrt{n}}(\Ne_2)=\lim_{q\rightarrow \infty}q\ln\pi_{q,\sqrt{n}}(\Ne_2^*)=\ln n+ \lim_{q\rightarrow \infty}q\ln\|\Ne_2^*:S_q^n\rightarrow M_n\|_{\sqrt{n}}.
\end{align*}By considering our previous estimates we easily deduce 
\begin{align}\label{cero condition}
-\ln 2\leq \lim_{q\rightarrow \infty}q\ln\|\Ne_2^*:S_q^n\rightarrow M_n\|_{\sqrt{n}}\leq 0.
\end{align}
Let $$\Ne:S_1^n\oplus_1 S_1^n\rightarrow S_1^n$$ be the quantum channel defined via $\Ne^*:M_n\rightarrow M_n\oplus_\infty M_n$ such that $$\Ne^*(A)=\Ne_1^*(A)\oplus \Ne_2^*(A).$$By considering the particular factorization $\Ne^*=\Ne^*\circ id_n$ in the factorization theorem, we have that 
\begin{align*}
\pi_{q,d}(\Ne^*)\leq n^{\frac{1}{q}}\|\Ne^*:S_q^n\rightarrow M_n\oplus_\infty M_n\|_d=n^{\frac{1}{q}}\max\Big\{\|\Ne_1^*:S_q^n\rightarrow M_n\|_d, \|\Ne_2^*:S_q^n\rightarrow M_n\|_d\Big\},
\end{align*}for every $q$ and $d$. Therefore,
\begin{align}\label{Equation 1 nonadditivity II}
C_{prod}^{\sqrt{n}}(\Ne)=\lim_{q\rightarrow \infty}q\ln\pi_{q,\sqrt{n}}(\Ne^*)\leq\lim_{q\rightarrow \infty}q\ln n^{\frac{1}{q}}\max\Big\{\|\Ne_1^*:S_q^n\rightarrow M_n\|_{\sqrt{n}}, \|\Ne_2^*:S_q^n\rightarrow M_n\|_{\sqrt{n}}\Big\}\\\nonumber=\ln n+ \lim_{q\rightarrow \infty}q \ln \max\{1, \|\Ne_2^*:S_q^n\rightarrow M_n\|_{\sqrt{n}}\}=\ln n,
\end{align}where we have used (\ref{cero condition}) in the last inequality. 

Finally, if we consider $\Ne\otimes \Ne:(S_1^n\oplus_1 S_1^n)\otimes (S_1^n\oplus_1 S_1^n)\rightarrow S_1^n\otimes S_1^n$ we see that this channel extends $\Ne_1\otimes \Ne_2$. Therefore, according to (\ref{Equation 1 nonadditivity I}) and (\ref{Equation 1 nonadditivity II}) we have
\begin{align*}
C_{prod}^n(\Ne\otimes \Ne)\geq C_{prod}^n(\Ne_1\otimes \Ne_2)\geq 2\ln n +\frac{1}{3}\ln n-\ln 2\geq 2C_{prod}^{\sqrt{n}}(\Ne)+\frac{1}{3}\ln n-\ln 2,
\end{align*}as we wanted. Note that the inequality $C_{prod}^n(\Ne\otimes \Ne)\geq C_{prod}^n(\Ne_1\otimes \Ne_2)$ follows from the fact that the norm 
\begin{align*}
\Big\|id\otimes (\Ne^*\otimes \Ne^*):S_q^n\otimes_{min} (M_n\otimes_{min}M_n)\rightarrow S_q^n\big[(M_n\oplus_\infty M_n)\otimes_{min}(M_n\oplus_\infty M_n)\big]\Big\|
\end{align*}is greater than or equal to
\begin{align*}
\Big\|id\otimes (\Ne_1^*\otimes \Ne_2^*):S_q^n\otimes_{min} (M_n\otimes_{min}M_n)\rightarrow S_q^n\big[M_n\otimes_{min}M_n\big]\Big\|,
\end{align*}which can be seen as a consequence of \cite[Corollary 1.3]{Pisierbook2} and it exactly means that $\pi_{q,d}(\Ne^*\otimes \Ne^*)\geq  \pi_{q,d}(\Ne_1^*\otimes \Ne_2^*)$ for every $q$ and $d$. This completes the proof.
\end{proof}
%
%
%%%%%%%%%%%%%%%%%%%%%%%%%%%%%%%%%%%%%%%%%%%%%%%
%
%
%
%
\section*{Acknowledgments}
We thank Toby S. Cubitt for many helpful conversations. A part of this work was done at the Institut Mittag-Leffler, Djursholm (Sweden), during a Semester on Quantum Information Theory in Fall 2010 .
\section{Appendix: Physical interpretation of the restricted capacities of quantum channels}\label{Phy Int}
In this section we will explain the notion of classical capacity of a quantum channel in more detail. In particular, we will state Holevo-Schumacher-Westmoreland's (HSW) Theorem, Bennett, Shor, Smolin, Thapliyal's (BSST) Theorem and we will explain the connections between the capacity studied in this work and the one studied in \cite{Shor}. 

Given a quantum channel $\Ne:S_1^n\rightarrow S_1^n$, the $d$-restricted classical capacities of the channel can be defined within the following common ratio-expression:
\begin{align*}
\lim_{\epsilon\rightarrow 0}\limsup_{k\rightarrow \infty}\Big\{\frac{m}{k}:\exists_\mathcal A,\exists_\mathcal B \text{   such that   } \big\|id_{\ell_1^{2^m}}-\mathcal B\circ \Ne^{\otimes_k}\circ \mathcal A\big\|< \epsilon\Big\}.
\end{align*}
Let us first assume $d=1$, so that Alice and Bob are not allowed to use any entanglement in their protocol to encode-transmit-decode information. Then, $\mathcal A:\ell_1^{2^m}\rightarrow \otimes^kS_1^n$ will be a quantum channel representing Alice's encoder from classical information to quantum information. On the other hand, Bob will decode the information he receives from Alice via the $k$ times uses of the channel, $\Ne^{\otimes_k}:\otimes^kS_1^n\rightarrow \otimes^kS_1^n$, by means of a quantum channel $\mathcal B:\otimes^kS_1^n\rightarrow\ell_1^{2^m}$. The key point here is that they want this composition to be asymptotically close to the identity map. That is, they want to have $\big\|id_{\ell_1^{2^m}}-\mathcal B\circ \Ne^{\otimes_k}\circ \mathcal A\big\|< \epsilon$.

The case in which Alice and Bob are allowed to share an entangled state is a bit more subtle. Let us assume that they share a $d$-dimensional state $\rho\in S_1^d\otimes S_1^d(=S_1(H_A)\otimes S_1(H_B)$) in the protocol. Then, a general encoder for Alice will be described by a channel of the form: $$\mathcal A=(\mathcal M_\mathcal A\otimes id_d)\circ i:\ell_1^{2^m}\rightarrow \ell_1^{2^m}\otimes (S_1^d\otimes S_1^d)\rightarrow \otimes^kS_1^n\otimes S_1^d,$$where $i:\ell_1^{2^m}\rightarrow \ell_1^{2^m}\otimes (S_1^d\otimes S_1^d)$ is the map defined by $i(x)=x\otimes \rho$ for every $x\in \ell_1^{2^m}$ and $\mathcal M_\mathcal A:\ell_1^{2^m}\otimes S_1^d\rightarrow \otimes^kS_1^n$ is a quantum channel. Since Alice has not access to Bob's part of $\rho$, the state received by Bob will be of the form $\big((\Ne^{\otimes_k}\otimes id_d)\circ \mathcal A\big)(x)$, where $x$ is the message that Alice wants to transmit. Finally, Bob's decoder will be a quantum channel $$\mathcal B:\otimes^kS_1^n\otimes S_1^d\rightarrow  \ell_1^{2^m}.$$
As in the previous case, the goal is to have $\big\|id_{\ell_1^{2^m}}-\mathcal B\circ (\Ne^{\otimes_k}\otimes id_d)\circ \mathcal A\big\|< \epsilon$.

The following diagram represents the two previous situations:
$$\xymatrix@R=1cm@C=2cm {{\bigotimes^k S_1^n}\ar[r]^{\Ne^{\otimes_k}} &
{\bigotimes^k S_1^n}\ar[d]^{\mathcal B} \\
{\ell_1^{2^m}}\ar[u]^{\mathcal A}\ar[r]^{\mathcal B\circ \Ne^{\otimes_k}\circ \mathcal A} & {\ell_1^{2^m}}}
\text{   }\text{   } \text{   }
\xymatrix@R=1cm@C=2cm {{\bigotimes^k S_1^n\otimes S_1^d}\ar[r]^{\Ne^{\otimes_k}\otimes id_d} &
{\bigotimes^k S_1^n\otimes S_1^d}\ar[dd]^{\mathcal B} \\
{(\ell_1^{2^m}\otimes S_1^d)\otimes S_1^d}\ar[u]^{\mathcal M_{\mathcal A}\otimes id_d} & {}\\
{\ell_1^{2^m}}\ar[u]^{i}\ar[r]^{\mathcal B\circ (\Ne^{\otimes_k} \otimes id_d)\circ \mathcal A} & {\ell_1^{2^m}}.}
$$
Note that, if we refer to ``dimension $d$ per channel use'', we should consider a state $\rho$ of dimension $d^k$ ($\rho\in S_1^{d^k}\otimes S_1^{d^k}$) if we are using $k$ times the channel ($\Ne^{\otimes_k}$) in our protocol.

Motivated by the noisy channel coding theorem (\ref{classical-classical- Equation}), one could try to obtain a similar result for the $d$-restricted capacities. However, the situation is more difficult in the quantum case (even for $d=1$). A first approach to the problem consists of restricting the protocols that Alice and Bob can perform. HSW theorem gives a nice formula when for the classical capacity of a quantum channel\footnote{When we don't mention assisted-entanglement means that no entanglement is allowed in the protocol.} Alice (the sender) is not allowed to distribute one entangled state among more than one channel use. That is, Alice encodes her messages into product states: $\rho_1\otimes\cdots \otimes \rho_k\in \otimes^kS_1^n$ (see \cite{Hol}, \cite{SW}).
\begin{theorem}[HSW]\label{Theorem Holevo}
Given a quantum channel $\Ne:S_1^n\rightarrow S_1^n$, we define
\begin{align*}
\chi(\Ne):= \sup \Big\{S\Big(\Ne(\sum_{i=1}^N\lambda_i\rho_i)\Big)-\sum_{i=1}^N\lambda_iS\Big(\Ne(\rho_i)\Big)\Big\},
\end{align*}where the supremum is taken over all $N\in \N$, all probability distributions $(\lambda_i)_{i=1}^N$ and all states $\rho_i\in S_1^n$ for all $i=1,\cdots, N$. Then, $\chi(\Ne)$ is the classical capacity of the channel when the sender is not allowed to distribute one entangled state among more than one channel use.
\end{theorem}
When we compute a capacity imposing this restriction we usually talk about \emph{the product state capacity of $\Ne$}. The product state classical capacity, $\chi$, is also called \emph{Holevo capacity}. Note that it does coincide with our definition of $C^1_{prod}$ in (\ref{d-restricted capacity}). It is not difficult to see that the classical capacity of a quantum channel $\Ne$ is the regularization version of $\chi(\Ne)$:
\begin{align}\label{regularization of chi}
C_c(\Ne)=\chi(\Ne)_{reg}:=\sup_k\frac{\chi(\Ne^{\otimes_k})}{k}.
\end{align}
It follows from (\ref{regularization of chi}) that $\chi(\Ne)\leq C_c(\Ne)$.  Whether $C_c(\Ne)=\chi(\Ne)$ for every quantum channel $\Ne$ was a major question in QIT for a long time. The problem was recently solved by Hastings, who showed that both capacities are different for certain channels (\cite{Has}). We refer (\cite{ASW}, \cite{BrHo}, \cite{FKM}) for a more complete explanation of the problem and some open questions in the area. Hastings' result says that we do need to consider the regularization (\ref{regularization of chi}) to compute the classical capacity of a quantum channel. Remarkably, for classical channels the product state version of the capacity is given by the formula in (\ref{classical-classical- Equation}). The reason to have the same expression for the general capacity $C_c$ is that such an expression is additive on classical channels: $C_c(\Ne_1\otimes \Ne_2)=C_c(\Ne_1)+C_c(\Ne_2)$. So one immediately obtains equality between both capacities by (\ref{regularization of chi}).

One could argue that the form of $\chi(\Ne)$ given in Theorem \ref{Theorem Holevo} does not look so much like an analogue formula to (\ref{classical-classical- Equation}). Recall that, if we denote by $H(X)$ the Shannon entropy of a random variable $X$, the \emph{mutual information} of two random variables $X$, $Y$ is defined as $H(X:Y)=H(X)+X(Y)- H(X,Y)$. Since in the quantum setting Shannon entropy is replaced by the von Neumann entropy of a quantum channel $S(\rho):= -tr(\rho \ln \rho)$\footnote{In quantum information theory von Neumann entropy is usually defined as $S(\rho):= -tr(\rho \log_2 \rho)$.}, the quantum generalization of the mutual information for a bipartite mixed state $\rho^{AB}\in S_1^n\otimes S_1^n$, which reduces to the classical mutual information when $\rho^{AB}$ is diagonal in the product  basis of the two subsystems, is $$S(\rho^A)+ S(\rho^B)-S(\rho^{AB}),$$where $\rho^A=(id\otimes tr)(\rho^{AB})$ and $\rho^B=(tr\otimes id)(\rho^{AB})$. Thus, the expression
\begin{align*}
\max_{\rho\in S_1(H_{A})\otimes S_1(H_{B})}\Big\{S(\rho^B)+S\big(\Ne(\rho^A)\big)-S\Big(\big(\Ne\otimes id_{B(H_B)}\big)(\rho)\Big)\Big\}
\end{align*}is a natural generalization of the classical channel's maximal input:output mutual information (\ref{classical-classical- Equation}) to the quantum case and it is equal to the classical capacity whenever $\Ne$ is a classical channel. However, it was shown in \cite{BSST} that this amount exactly describes the (unlimited) entangled-assisted classical capacity $C_{E}(\Ne)$.
\begin{theorem}[BSST]\label{Theorem BSST}
For a noisy quantum channel $\Ne:S_1^n\rightarrow S_1^n$ the (unlimited) entanglement-assisted classical capacity is given by the expression:
\begin{align*}
C_{E}(\Ne)=\max_{\rho\in S_1^n\otimes S_1^n}\Big\{S\Big(\big(tr_n\otimes id_n\big)(\rho)\Big)+S\Big(\Ne(\big(id_n\otimes tr_n\big)(\rho)\Big)-S\big((\Ne\otimes id_n)(\rho)\big)\Big\}.
\end{align*}
\end{theorem}
There is a crucial difference between Theorem \ref{Theorem Holevo} and Theorem \ref{Theorem BSST}. While the first case describes the product state classical capacity, the last theorem describes the general (unlimited) entanglement-assisted classical capacity. In fact, it can be seen that the expression $C_{E}(\Ne)$ in Theorem \ref{Theorem BSST} describes the product state version of the capacity but, furthermore, it is additive on quantum channels. So, no regularization is required in this case. In \cite{Shor} the author studied the classical capacity of a quantum channel with restricted assisted entanglement (which involves, in particular, the two previous capacities). The main result proved in \cite{Shor} can be read as follows.
\begin{theorem}\label{interpretation Shor}
Given a quantum channel $\Ne:S_1^n\rightarrow S_1^n$, for any $1\leq d\leq \ln n$ we define
\begin{align}\label{d-restricted capacity Shor}
R_{prod}^d(\Ne):=\sup\Big\{S\Big(\sum_{i=1}^N \lambda_i \Ne(\rho_i)\Big) +\sum_{i=1}^N \lambda_i\Big[S\big(\rho_i\big)-S\big(\big(id_n\otimes \Ne\big)(\chi_{\rho_i})\Big)\Big] \Big\},
\end{align}where $\chi_{\rho_i}\in S_1^n\otimes S_1^n$ denotes any purification of the state $\rho_i$\footnote{For every state $\rho\in S_1^n$ there exists a unit element $x$ in $\ell_2^k\otimes \ell_2^n$ so that the pure state $\chi_\rho\in S_1^k\otimes S_1^n=S_1(\ell_2^k\otimes \ell_2^n)$ given by the rank-one projection on $x$, verifies that $(tr_k\otimes id_n)(\chi_\rho)=\rho$. Moreover, if we allow $\chi_\rho\in S_1^n\otimes S_1^n$ we can get $(tr_n\otimes id_n)(\chi_\rho)=(id_n\otimes tr_n)(\chi_\rho)=\rho$.}. Here, the supremum runs over all $N\in \N$, all probability distributions $(\lambda_i)_{i=1}^N$, and all families $(\rho_i)_{i=1}^N$ of states in $S_1^n$ with $\sum_{i=1}^N\lambda_iS(\rho_i)\leq d$. Then, $R_{prod}^d(\Ne)$ is the classical capacity of $\Ne$ with assisted entanglement when
\begin{enumerate}
\item[a)] Alice and Bob are restricted to protocols in which they can use entropy of entanglement $d$ per channel use.
\item[b)] The sender is not allowed to distribute one entangled state among more than one channel use.
\end{enumerate}
\end{theorem}
Indeed, although the aim of the work \cite{Shor} is to study the capacity $R^d(\Ne)$, where one has to remove restriction b) above, the main result presented in Shor's work is  the construction of a protocol between Alice and Bob with capacity equal to (\ref{d-restricted capacity Shor}) and verifying conditions a) and b) (see \cite[Theorem 1]{Shor}). Furthermore, Shor proved that the expression (\ref{d-restricted capacity Shor}) is an upper bound for the capacity $R^d(\Ne)$ if we add the restriction b) (see \cite[Section 4]{Shor}). In fact, one can follow the argument in \cite{Shor} verbatim to obtain the expression (\ref{d-restricted capacity}) for the capacity considered in this work, $C_{prod}^d(\Ne)$. 

Let us comment some differences between the capacities $R_{prod}^d(\Ne)$ and $C_{prod}^d(\Ne)$. First of all, the restriction a) in the previous theorem refers to the entropy of entanglement rather than to the dimension of the entanglement as in $C_{prod}^d(\Ne)$. Moreover, conditions a) and b) in $R_{prod}^d(\Ne)$ restrict to protocols in which, if Alice and Bob use $n$ times the channel, $\Ne^{\otimes _n}$, in their protocol, then they can use states of the form $\eta=\eta_1\otimes \cdots \otimes \eta_n$, where $\eta_i$ is a bipartite pure state for every $i$, so that $S_e(\eta)\leq nd$\footnote{$S_e(\rho):=S((id\otimes tr)(\rho))=S((tr\otimes id)(\rho))$ is the entropy of entanglement of the bipartite pure state $\rho$.}. This is slightly more general than imposing $S_e(\eta_i)\leq d$ for every $i$. This is reflected in the optimization condition $\sum_{i=1}^N\lambda_iS(\eta_i)\leq d$. However, in the definition of $C_{prod}^d(\Ne)$, we impose that $\eta_i\in S_1^d\otimes S_1^d$ for every $i$. This is the reason because condition a) is not completely analogous in the expression (\ref{d-restricted capacity}) and in Theorem \ref{interpretation Shor}. Moreover, in the definition of $R_{prod}^d(\Ne)$ the sender is allowed to input more than one state in the same channel at the price of using some of the channels without entanglement. This is not allowed in the definition of $C_{prod}^d(\Ne)$. This is the reason because condition b) is slightly different in the definition of $R_{prod}^d(\Ne)$ and $C_{prod}^d(\Ne)$. On the other hand, the fact that the expression in  (\ref{d-restricted capacity}) is not completely analogous to the one in (\ref{d-restricted capacity Shor}) is because the first one is expressed in the ``tensor product form''. However, one can easily show the following result.
\begin{theorem}\label{Shor's Expression}
Given a quantum channel $\Ne:S_1^n\rightarrow S_1^n$, we have
\begin{align}\label{shor's Expression formula}
C_{prod}^d(\Ne)=\tilde{C}_{prod}^d(\Ne):=\sup\Big\{S\Big(\sum_{i=1}^N \lambda_i\big(\Ne\circ \phi_i\big)(\delta_i)\Big)
\\\nonumber+\sum_{i=1}^N \lambda_i\Big [S(\delta_i)-S\Big(\big(id_d\otimes (\Ne\circ \phi_i)\big)(\chi_{\delta_i})\Big)\Big] \Big\}.
\end{align}Here, the supremum runs over all $N\in \N$, all probability distributions $(\lambda_i)_{i=1}^N$, all states $\delta_i\in S_1^d$ and all quantum channels $\phi_i:S_1^d\rightarrow S_1^n$ for every $i=1,\cdots, N$. $\chi_{\delta_i}$ denotes any purification of $\delta_i$ for every $i=1,\cdots, N$.
\end{theorem}
\begin{proof}
Let us start by showing $C_{prod}^d(\Ne)\leq \tilde{C}_{prod}^d(\Ne)$. For this, let us consider an ensemble $\big\{(\lambda_i)_{i=1}^N, (\eta_i)_{i=1}^N, (\phi_i)_{i=1}^N\big\}$  optimizing (\ref{d-restricted capacity}). Then, by defining $\delta_i=(tr_d\otimes id_d)(\eta_i)$ for every $i$, we can consider the ensemble $\big\{(\lambda_i)_{i=1}^N, (\delta_i)_{i=1}^N, (\phi_i)_{i=1}^N\big\}$.  By definition, $\eta_i$ is a purification of $\delta_i$, which implies, in particular,  $$S(\delta_i)=S\Big((tr_d\otimes id_d)(\eta_i)\Big)=S\Big((id_d\otimes tr_d)(\eta_i)\Big)$$for every $i$ (see for instance \cite[Section 2.5]{NC}). Then, the inequality follows by plugging the ensemble $\big\{(\lambda_i)_{i=1}^N, (\delta_i)_{i=1}^N, (\phi_i)_{i=1}^N\big\}$ in the expression of $\tilde{C}_{prod}^d(\Ne)$ above. For the converse inequality, $C_{prod}^d(\Ne)\geq \tilde{C}_{prod}^d(\Ne)$, given an ensemble $\big\{(\lambda_i)_{i=1}^N, (\delta_i)_{i=1}^N, (\phi_i)_{i=1}^N\big\}$ optimizing $\tilde{C}_{prod}^d(\Ne)$, we consider $\big\{(\lambda_i)_{i=1}^N, (\eta_i)_{i=1}^N, (\phi_i)_{i=1}^N\big\}$, where here $\eta_i\in S_1^n\otimes S_1^n$ is any purification of $\delta_i$ for every $i$. Again, the inequality follows by plugging the ensemble $\big\{(\lambda_i)_{i=1}^N, (\delta_i)_{i=1}^N, (\phi_i)_{i=1}^N\big\}$ in the expression (\ref{d-restricted capacity}) for $C_{prod}^d(\Ne)$.
\end{proof}
This theorem gives us an expression for $C_{prod}^d(\Ne)$ analogous to the expression in (\ref{Shor's Expression}). Since $S(\rho)\leq \ln d$ for every $d$-dimensional bipartite state, it is clear that $C_{prod}^d(\Ne)\leq R_{prod}^{\ln d}(\Ne)$ for every quantum channel $\Ne$. Furthermore, following the spirit of \cite{Shor} we may consider the same product state capacity as a function of the number $d$ of singlets\footnote{The singlet is the basic $2$-dimensional bipartite quantum state defined as the rank one projection onto the vector $\frac{1}{\sqrt{2}}(e_1\otimes e_1+e_2\otimes e_2)\in \ell_2^2\otimes \ell_2^2$.} per channel use, $E^d$, and we would trivially obtain $$E_{prod}^{\ln d}(\Ne)\leq C_{prod}^d(\Ne)\leq R_{prod}^{\ln d}(\Ne).$$In order to obtain the general capacities (rather than the product state version) one has to consider the regularization. It is not difficult to see that in this case the regularization is given by
\begin{align*}
C^d(\Ne)=\sup_k\frac{C_{prod}^{d^k}(\otimes^k \Ne)}{k}
\end{align*}and analogously for $R^d$ and $E^d$. Interestingly, using the explicit form of the protocol given by Shor in \cite{Shor} and the fact that entanglement is an interconvertible resource (see \cite{LP}), one can conclude (see \cite[Section 5]{Shor}) that $$E^{\ln d}(\Ne)=C^d(\Ne)=R^{\ln d}(\Ne).$$
That is, the three capacities $E_{prod}^{\ln d}$, $C_{prod}^d$ and $R_{prod}^{\ln d}$ represent different product state versions of the same capacity. Theorem \ref{counterexample additivity} tells us that we do need to consider their regularization since $C^d(\Ne)$ can be very different from $C_{prod}^d(\Ne)$. In fact, Theorem \ref{counterexample additivity} can be proved for $R_{prod}^{\ln d}(\Ne)$ and $E_{prod}^{\ln d}(\Ne)$ by using similar ideas.

We should remark here that the capacity $C^d$ has been also studied in some recent works (see \cite{HsWi}, \cite{WiHs} and the references therein). There, the authors study the communication rates of a quantum channel when combined with the noiseless resources of classical communication, quantum communication and entanglement. However, the approach in those works is different from the one followed in this paper and they do not consider the product state capacity $C_{prod}^d(\Ne)$.

Finally, note that the restriction in the entanglement dimension considered in this work, ``implies'' that in our formulae (\ref{d-restricted capacity}) and (\ref{Shor's Expression}) we need to optimize also over quantum channels $(\phi_i:S_1^d\rightarrow S_1^n)_i$. This is because we need to encode our initial states, in dimension $d$, into states of the same dimension of the channel input ($n$ in our case). We finish this section by proving that such a bother can be avoided.
\begin{prop}\label{Shor's Expression II}
Given a quantum channel $\Ne:S_1^n\rightarrow S_1^n$, we have
\begin{align}\label{shor's Expression formula II}
C_{prod}^d(\Ne)=\tilde{\tilde{C}}^d_{prod}(\Ne):=sup\Big\{S\Big(\sum_{i=1}^N \lambda_i \Ne((tr_d\otimes id_n)(\rho_i))\Big)
\\ \nonumber+\sum_{i=1}^N \lambda_i\Big[S\Big((id_d\otimes tr_n)(\rho_i)\Big)-S\Big(\big(id_d\otimes \Ne\big)(\rho_i)\Big)\Big] \Big\}.
\end{align}Here, the supremum runs over all $N\in \N$, all probability distributions $(\lambda_i)_{i=1}^N$ and all families $(\rho_i)_{i=1}^N$ of quantum states $\rho_i\in S_1^d\otimes S_1^n$ for every $i=1,\cdots, N$.
\end{prop}
The proof of Proposition \ref{Shor's Expression II} is based on the following lemma.
\begin{lemma}\label{xieta}
Let $\xi\in S_1^m \otimes S_1^d$ be a state and let $\gamma\in S_1^k \otimes S_1^d$ be a purification of $\eta=(tr_m\ten id_d)(\xi)$. Then, there is a completely positive and trace preserving map $\Phi:S_1^k\to S_1^m$ such that
 \[ (\Phi\ten id_d)(\gamma) \lel \xi \pl .\]
\end{lemma}
\begin{proof}
Let $\xi\in S_1^m \otimes S_1^d$ and  $\eta=(tr_m\ten id_d)(\xi)$ be as in the statement of the lemma. Then, we can invoke the Schmidt decomposition to find a probability distribution $(\lambda_k)_{k=1}^d$ and an orthonormal basis $(e_k)_{k=1}^d$ of $\ell_2^d$ such that 
$\eta=\sum_{k=1}^d\lambda_ke_{k,k}$\footnote{We can assume that $\eta$ has rank $d$. Otherwise, we can realize $\xi$ as an element of $S_1^m \otimes S_1^k$ with $k=rank(\eta)$ and exactly the same proof works. $e_{k,l}$ must be understood as the corresponding matrix written in the basis $(e_k)_{k=1}^d$.}. Let us write $\xi=\sum_{i,j=1}^d\xi_{i,j}\otimes e_{i,j}$ and note that
\begin{align}\label{obvious schimdt decomp}
\eta=\sum_{k=1}^d\lambda_ke_{k,k}=\sum_{i,j=1}^dtr_m(\xi_{i,j})e_{i,j}.
\end{align}The smallest purification of $\eta$ is given by $$\gamma=\sum_{i,j=1}^d\sqrt{\lambda_i\lambda_j}e_{i,j}\otimes e_{i,j}.$$Then, we define the linear map $\Phi:S_1^d\to S_1^m$ given by $$\Phi(e_{i,j})=\frac{1}{\sqrt{\lambda_i\lambda_j}}\xi_{i,j}.$$With this definition we clearly have
$$(\Phi\ten id_d)(\gamma) =\sum_{i,j=1}^d\xi_{i,j}\otimes e_{i,j}=\xi.$$
Therefore, we need to show that $\Phi$ is completely positive and trace preserving. We will show, equivalently, that $\Phi^*:M_m\rightarrow M_d$ is completely positive and unital. To this end, we claim that
\begin{align}\label{claim lemma cp}
\Phi^*(b)=(\lambda^{-1/2})^*T_\xi (b)\lambda^{-1/2}
\end{align}for every $b\in M_m$, where here $T_\xi :M_m\rightarrow M_d$ denotes the linear map associated to the tensor $\xi$ and  $\eta^{-1/2}$ denotes the row matrix $(\lambda_1^{-1/2},\cdots, \lambda_d^{-1/2})$. As we have mentioned before, the positivity of $\xi$ is equivalent to the completely positivity of  the map $T_\xi$. Then, one can immediately conclude from (\ref{claim lemma cp}) that  $\Phi^*$ is completely positive. On the other hand,
\begin{align*}
\Phi^*(\uno_m)=(\lambda^{-1/2})^*T_\xi (\uno_m)\lambda^{-1/2}&=(\lambda^{-1/2})^*(tr_m\otimes id_d)(\xi(\uno_m\otimes \uno_d))\lambda^{-1/2}\\&=(\lambda^{-1/2})^*(tr_m\otimes id_d)(\xi)\lambda^{-1/2}=\uno_d,
\end{align*}where in the last equality we have used (\ref{obvious schimdt decomp}). Thus, $\Phi^*$ is completely positive and unital, as we wanted.

It remains to prove our claim (\ref{claim lemma cp}). To this end, we fist note that
$$\Phi^*(b)=\sum_{i,j=1}^d\frac{1}{\sqrt{\lambda_i\lambda_j}}tr(\xi_{i,j}b^{tr})e_{i,j}=(\lambda^{-1/2})^*\Big(\sum_{i,j=1}^dtr(\xi_{i,j}b^{tr})e_{i,j}\Big)\lambda^{-1/2}.$$Indeed, $$\langle\Phi^*(b), e_{k,l}\rangle=\langle b , \Phi(e_{k,l})\rangle=\frac{1}{\sqrt{\lambda_k\lambda_l}}tr(b \xi_{k,l}^{tr})=\frac{1}{\sqrt{\lambda_k\lambda_l}}tr(b^{tr} \xi_{k,l}).$$
On the other hand,
$$T_\xi (b)=(tr_m\otimes id_d)\big(\xi(b^{tr}\otimes \uno_d)\big)=\sum_{i,j=1}^dtr_m(\xi_{i,j}b^{tr})e_{i,j}.$$So we obtain (\ref{claim lemma cp}).

Finally, let us assume that $\gamma_h\in S_1(H)\otimes S_1^d$ is any other purification of $\eta$, where $h\in H\otimes \ell_2^d$. We can write $h=\sum_{k=1}^dh_k\otimes e_k$. Thus,$$\big(tr_{S_1(H)}\otimes id_d\big)(\gamma_h)=\sum_{i,j=1}^d\langle h_i, h_j\rangle e_{i,j}=\sum_{k=1}^d\lambda_k e_{k,k}.$$This trivially implies that $\langle h_i, h_j\rangle=\lambda_i\delta_{i,i}$ for every $i,j$. Hence, we can define the linear isometry $v:\ell_2^d\rightarrow H$ given by $v(e_k)=\tilde{h}_k:=\frac{1}{\sqrt{\lambda_k}}h_k$ for every $k$. Moreover, we will consider the orthogonal projection $p:H\rightarrow v(\ell_2^d)$. Then, we define $\Psi:S_1(H)\rightarrow S_1^d$, given $$\Psi(a)=v^*av+tr_{S_1(H)}\big((1-p)a(1-p)\big)e_{1,1}$$for every $a$. Then, it is very easy to see that $ \Psi$ is completely positive and trace preserving and, moreover, $$(\Psi\otimes id_d)(\gamma_h)=\sum_{i,j=1}^d\sqrt{\lambda_i\lambda_j}e_{i,j}\otimes e_{i,j}.$$ Therefore, $\Phi \Psi$ does the job.
\end{proof}
\begin{proof}[Proof of Proposition \ref{Shor's Expression II}]

Inequality $C_{prod}^d(\Ne)\leq \tilde{\tilde{C}}^d_{prod}(\Ne)$ can be proved very easily. Indeed, given an optimal ensemble for ${C}^d(\Ne)$, $\big\{(\lambda_i)_{i=1}^N, (\eta_i)_{i=1}^N, (\phi_i)_{i=1}^N\big\}$, we just need to consider the ensemble $\big\{(\lambda_i)_{i=1}^N, (\rho_i)_{i=1}^N, \big\}$, where  $\rho_i=(id_d\otimes \phi_i)(\eta_i)$ for every $i$ for the optimization in (\ref{shor's Expression formula II}). In order to prove inequality $C_{prod}^d(\Ne)\geq \tilde{\tilde{C}}^d_{prod}(\Ne)$, we consider again an optimal ensemble $\big\{(\lambda_i)_{i=1}^N, (\rho_i)_{i=1}^N, \big\}$ for $\tilde{\tilde{C}}^d_{prod}(\Ne)$. Then, we define $\delta_i=(id_d\otimes tr_n)(\rho_i)\in S_1^d$ and $\eta_i=\chi_{\delta_i}\in S_1^d\otimes S_1^d$ for every $i$. According to Lemma \ref{xieta}, we can find quantum channels $\phi_i:S_1^d\rightarrow S_1^n$ so that $(id_d\otimes \phi_i)(\eta_i)=\rho_i$ for every $i$. Then, we obtain the desired inequality by noting 
\begin{align*}
S\Big(\sum_{i=1}^N \lambda_i \Ne\big((tr_d\otimes id_n)(\rho_i)\big)\Big)&=S\Big(\sum_{i=1}^N \lambda_i \Ne\big((tr_d\otimes id_n)\big((id_d\otimes \phi_i)(\eta_i)\big)\big)\Big)\\&=S\Big(\sum_{i=1}^N \lambda_i (\Ne\circ \phi_i)\big((tr_d\otimes id_d)(\eta_i)\big)\Big),
\end{align*}
\begin{align*}
\sum_{i=1}^N \lambda_iS\Big((id_d\otimes tr_n)(\rho_i)\Big)&=\sum_{i=1}^N \lambda_iS\Big((id_d\otimes tr_n)\big((id_d\otimes \phi_i)(\eta_i)\big)\Big)\\&=\sum_{i=1}^N \lambda_iS\Big((id_d\otimes tr_d)(\eta_i)\Big),
\end{align*}
\begin{align*}
\sum_{i=1}^N \lambda_iS\Big(\big(id_d\otimes \Ne\big)(\rho_i)\Big)&=\sum_{i=1}^N \lambda_iS\Big(\big(id_d\otimes \Ne\big)\big((id_d\otimes \phi_i)(\eta_i)\big)\Big)\\&=\sum_{i=1}^N \lambda_iS\Big(\big(id_d\otimes( \Ne\circ \phi_i)\big)(\eta_i)\Big).
\end{align*}\end{proof}
\vskip 2cm
\hfill \noindent \textbf{Marius Junge} \\
\null \hfill Department of Mathematics \\ \null \hfill
University of Illinois at Urbana-Champaign\\ \null \hfill 1409 W. Green St. Urbana, IL 61891. USA
\\ \null \hfill\texttt{junge@math.uiuc.edu}

\vskip 1cm

\hfill \noindent \textbf{Carlos Palazuelos} \\
\null \hfill Instituto de Ciencias Matem\'aticas\\ \null \hfill
CSIC-UAM-UC3M-UCM \\ \null \hfill Consejo Superior de
Investigaciones Cient{\'\i}ficas \\ \null \hfill C/ Nicol\'as Cabrera 13-15.
28049, Madrid. Spain \\ \null \hfill\texttt{carlospalazuelos@icmat.es}
\end{document}